\documentclass[a4paper,12pt]{article}

\usepackage{amscd}
\usepackage{amssymb,amsfonts,amsmath,amsthm}
\usepackage{verbatim}

\usepackage{amsmath}
\usepackage{amsthm}
\usepackage{amssymb}

\usepackage{latexsym}
\usepackage{array}
\usepackage{enumerate}

\numberwithin{equation}{section}

\usepackage{amsmath,amssymb,mathrsfs,amsthm}
\usepackage[all]{xy}
\usepackage{graphicx}
\usepackage{ulem}
\usepackage{ascmac}

\usepackage{mathtools}

\newcommand{\BDC}{{\mathbf{D}}^{\mathrm{b}}}

\newcommand{\Mod}{\mathrm{Mod}}

\newcommand{\shom}{{\mathcal{H}}om}

\newcommand{\CC}{\mathbb{C}}

\newcommand{\RR}{\mathbb{R}}

\newcommand{\ZZ}{\mathbb{Z}}
\newcommand{\D}{\mathcal{D}}

\newcommand{\F}{\mathcal{F}}

\newcommand{\M}{\mathcal{M}}

\renewcommand{\(}{\left(}
\renewcommand{\)}{\right)}

\newcommand{\Image}{\operatorname{Im}}

\newcommand{\an}{{\rm an}}

\newcommand{\supp}{{\rm supp}}

\newcommand{\Perv}{{\rm Perv}}
\newcommand{\Sol}{{\rm Sol}}

\newcommand{\sub}{{\rm sub}}

\newcommand{\tl}[1]{\widetilde{#1}}

\newcommand{\simto}{\overset{\sim}{\longrightarrow}}

\newcommand{\op}{{\mbox{\scriptsize op}}}

\newcommand{\SO}{\mathcal{O}}

\newcommand{\SL}{\mathcal{L}}

\newcommand{\SF}{\mathcal{F}}

\newcommand{\SH}{\mathcal{H}}

\newcommand{\Modhol}{\mathrm{Mod}_{\mbox{\rm \scriptsize hol}}}
\newcommand{\Modrh}{\mathrm{Mod}_{\mbox{\rm \scriptsize rh}}}

\newcommand{\BDChol}{{\mathbf{D}}^{\mathrm{b}}_{\mbox{\rm \scriptsize hol}}}
\newcommand{\BDCrh}{{\mathbf{D}}^{\mathrm{b}}_{\mbox{\rm \scriptsize rh}}}

\newcommand{\DChol}{\mathbf{D}_{\mbox{\rm \scriptsize hol}}}

\newcommand{\rhom}{{\bfR}{\mathcal{H}}om}

\newcommand{\I}{{\rm I}}

\newcommand{\var}[1]{\overline{#1}}
\newcommand{\BEC}{{\mathbf{E}}^{\mathrm{b}}}

\newcommand{\ZEC}{{\mathbf{E}}^{\mathrm{0}}}

\newcommand{\T}{{\mathsf{T}}}
\newcommand{\bfR}{\mathbf{R}}

\newcommand{\bfD}{\mathbf{D}}

\newcommand{\rmE}{{\rm E}}
\newcommand{\rmD}{{\rm D}}

\newcommand{\bfE}{\mathbf{E}}

\newcommand{\sh}{{\rm sh}}
\newcommand{\Ish}{{\rm Ish}}

\newtheorem{theorem}{Theorem}[section]

\newtheorem{corollary}[theorem]{Corollary}
\newtheorem{lemma}[theorem]{Lemma}

\newtheorem{proposition}[theorem]{Proposition}

\newtheorem{notation}[theorem]{Notation}
\theoremstyle{definition}
\newtheorem{definition}[theorem]{Definition}
\theoremstyle{remark}
\newtheorem{remark}[theorem]{\sc Remark}

\title{Enhanced Perverse Subanalytic Sheaves
}

\author{ Yohei ITO
\footnote{Department of Mathematics, Faculty of Science Division II,
Tokyo University of Science, 1-3, Kagurazaka, Shinjuku-ku, Tokyo, 162-8601, Japan. E-mail: yitoh@rs.tus.ac.jp}}

\date{}

\begin{document}
\maketitle

\begin{flushright}
\textit{Dedicated to Professor Toshiyuki Kobayashi,\\
with gratitude and respect}
\end{flushright}


\begin{abstract}
In \cite{Ito24a}, the author explained a relation between enhanced ind-sheaves and enhanced subanalytic sheaves.
In particular, a relation between \cite[Thm.\:9.5.3]{DK16} and \cite[Thm.\:6.3]{Kas16} had been explained.
Moreover, in \cite{Ito24b},
the author defined $\CC$-constructibility for enhanced subanalytic sheaves
and proved that there exists an equivalence of categories
between the triangulated category of holonomic $\D$-modules and
that of $\CC$-constructible enhanced subanalytic sheaves.
In this paper, we will show that there exists a t-structure
on the triangulated category of $\CC$-constructible enhanced subanalytic sheaves 
whose heart is equivalent to the abelian category of holonomic $\D$-modules.
Furthermore, we shall consider simple objects of its heart and minimal extensions of objects of its heart.
\end{abstract}

\section{Introduction}
\subsection{Riemann--Hilbert Problem}
The original Riemann--Hilbert problem asks for the existence of a linear ordinary differential equation with regular singularities and a given monodromy on a curve.

In \cite{Del}, P.\:Deligne formulated it as a correspondence between
meromorphic connections on a complex manifold $X$ with regular singularities along a hypersurface $Y$
and local systems on $X\setminus Y$.

Moreover, M.\:Kashiwara extended Deligne's correspondence as a correspondence
between regular holonomic $\D_X$-modules and $\CC$-constructible sheaves on $X$ as below.

\subsection{Regular Riemann--Hilbert Correspondence}
Let $X$ be a complex manifold and $\D_X$ the sheaf of rings of holomorphic differential operators on $X$.

Let us denote by $\SO_X$ the sheaf of rings of holomorphic functions on $X$
and denote by $\rhom_{\D_X}(\cdot, \cdot)$ the right derived functor
of the hom-functor $\shom_{\D_X}(\cdot, \cdot)$.
Moreover, we denote by $\BDCrh(\D_X)$ the triangulated category of regular holonomic $\D_X$-modules
and by $\BDC_{\CC\mbox{\scriptsize -}c}(\CC_X)$ that of $\CC$-constructible sheaves on $X$.
Then we have an equivalence of categories which is called
the regular Riemann--Hilbert correspondence or Riemann--Hilbert correspondence for regular holonomic $\D$-modules.
\begin{theorem}[{\cite[Main Theorem]{Kas84}}]\label{thm_regRH}
There exists an equivalence of categories:
$$\Sol_X\colon \BDCrh(\D_X)^\op\simto \BDC_{\CC\mbox{\scriptsize -}c}(\CC_X),\
\M\mapsto 
\rhom_{\D_X}(\M, \SO_X).$$
\end{theorem}

After the appearance of the regular Riemann--Hilbert correspondence,
A.\:Beilinson and J.\:Bernstein developed systematically
a theory of regular holonomic $\D$-modules on smooth algebraic varieties
over the complex number field $\CC$ and
obtained an algebraic version of the regular Riemann--Hilbert correspondence
which is called the algebraic regular Riemann--Hilbert correspondence.
See \cite{Be, Bor} and also \cite{Sai} for the details.

\subsection{Perverse Sheaves}
On the other hand, motivated by applications for arithmetic geometry,
A.\:A.\:Beilinson, J.\:Bernstein and P.\:Deligne introduced the notions
of perverse t-structures and perverse sheaves \cite{BBD}.
Note that the triangulated category $\BDC_{\CC\mbox{\scriptsize -}c}(\CC_X)$ has the perverse t-structure
$\big({}^p\bfD^{\leq0}_{\CC\mbox{\scriptsize -}c}(\CC_X),
{}^p\bfD^{\geq0}_{\CC\mbox{\scriptsize -}c}(\CC_X)\big).$
Let us denote by $$\Perv(\CC_X) :=
{}^p\bfD^{\leq0}_{\CC\mbox{\scriptsize -}c}(\CC_X)\cap {}^p\bfD^{\geq0}_{\CC\mbox{\scriptsize -}c}(\CC_X)$$ its heart
and call an object of $\Perv(\CC_X)$ a perverse sheaf.
Surprisingly, the functor $\Sol_X$ induces an equivalence of categories between
the abelian category $\Modrh(\D_X)$ of regular holonomic $\D_X$-modules
and $\Perv(\CC_X)$:
$$\Sol_X(\cdot)[d_X]\colon \Modrh(\D_X)^\op\simto \Perv(\CC_X).$$
Here, $d_X$ is the complex dimension of $X$.
See \cite[Thm.\:4.1]{Kas75} (also \cite[Thm.\:3.5.1]{Bjo93}) for the details of the analytic case,
\cite[Thm.\:22.4]{Bor} for the details of the algebraic case.

By using this equivalence of categories,
Brylinski--Kashiwara\cite{BK} and Beilinson--Bernstein\cite{BB} independently solved 
the Kazhdan--Lusztig conjecture,
which was a great breakthrough in representation theory. 

\subsection{Irregular Riemann--Hilbert Correspondence and Enhanced Ind-Sheaves}
The problem of extending
the Riemann--Hilbert correspondence to cover
the case of holonomic $\D$-modules with irregular singularities
had been open for 30 years.
After a groundbreaking development in the theory of irregular meromorphic connections by 
K.\:S.\:Kedlaya \cite{Ked10, Ked11} and T.\:Mochizuki \cite{Mochi09, Mochi11},
A.\:D'Agnolo and M.\:Kashiwara established the Riemann--Hilbert correspondence
for irregular holonomic $\D$-modules in \cite{DK16} as follows.

For this purpose, they introduced enhanced ind-sheaves
extending the notion of ind-sheaves introduced by M.\:Kashiwara and P.\:Schapira in \cite{KS01}.
Let us denote by $\BDChol(\D_{X})$ the triangulated category of
holonomic $\D_X$-modules and by $\BEC_{\RR\mbox{\scriptsize -}c}(\I\CC_X)$
that of $\RR$-constructible enhanced ind-sheaves on $X$.
See \cite[Def.\:4.9.2]{DK16} for the definition of $\RR$-constructible enhanced ind-sheaves.
Then we have:

\begin{theorem}[{\cite[Thm.\:9.5.3, Prop.\:9.5.4]{DK16}}]\label{irregRH_DK}
There exists a fully faithful functor:
\[\Sol_X^{\rmE} \colon \BDChol(\D_X)^{\op}\hookrightarrow\BEC_{\RR\mbox{\scriptsize -}c}(\I\CC_X)\]
and the following diagram is commutative:
\[\xymatrix@R=25pt@C=60pt@M=7pt{
\BDChol(\D_X)^{\op}\ar@{^{(}->}[r]^-{\Sol_X^{\rmE}}
 & \BEC_{\RR\mbox{\scriptsize -}c}(\I\CC_X)\\
\BDCrh(\D_X)^{\op}\ar@{->}[r]_-{\Sol_X}^-{\sim}\ar@{}[u]|-{\bigcup}
&\BDC_{\CC\mbox{\scriptsize -}c}(\CC_X).
\ar@{^{(}->}[u]_-{e_X\circ \iota_X}
}\]
\end{theorem}
Here, $e_X\circ \iota_{X}\colon \BDC(\CC_X) \to \BEC(\I\CC_X)$ is the natural embedding functor
from the derived category $\BDC(\CC_X)$ of sheaves of $\CC$-vector spaces on $X$ to
the triangulated category $\BEC(\I\CC_{X})$ of enhanced ind-sheaves on $X$.
See \cite[Prop.\:4.7.15]{DK16} for the details.
See also \cite[Def.\:9.1.1]{DK16} for the definition of $\Sol_X^\rmE$.

Furthermore, T.\:Mochizuki proved that
the essential image of $\Sol_X^{\rmE}$ can be characterized  by the curve test \cite{Mochi22}.
Remark that T.\:Kuwagaki \cite[Thm.\:8.6]{Kuwa21} introduced another approach
to the irregular Riemann--Hilbert correspondence
via irregular constructible sheaves which are defined by $\CC$-constructible sheaves
with coefficients in a finite version of the Novikov ring and special gradings.

\subsection{$\CC$-Constructible Enhanced Ind-Sheaves}
In \cite{Ito20}, the author defined $\CC$-constructibility for enhanced ind-sheaves on $X$
and proved that they are nothing but objects of the essential image of $\Sol_X^{\rmE}$.
Namely, we obtain an equivalence of categories as below.
We denote by $\BEC_{\CC\mbox{\scriptsize -}c}(\I\CC_X)$
the triangulated category of $\CC$-constructible enhanced ind-sheaves on $X$.
See \cite[Def\:3.19]{Ito20} for the definition of  $\CC$-constructible enhanced ind-sheaves.

\begin{theorem}[{\cite[Thm.\:3.26]{Ito20}, \cite[Prop.\:3.1]{Ito23}}]\label{irregRH_ana}
The functor $\Sol_X^\rmE$ induces an equivalence of categories:
\[\Sol_X^{\rmE} \colon \BDChol(\D_X)^{\op}\simto \BEC_{\CC\mbox{\scriptsize -}c}(\I\CC_X)\]
and the following diagram is commutative:
\[\xymatrix@R=25pt@C=60pt@M=7pt{
\BDChol(\D_X)^{\op}\ar@{->}[r]^\sim\ar@<1.0ex>@{}[r]^-{\Sol_X^{\rmE}}
 & \BEC_{\CC\mbox{\scriptsize -}c}(\I\CC_X)\\
\BDCrh(\D_X)^{\op}\ar@{->}[r]_-{\Sol_X}^-{\sim}\ar@{}[u]|-{\bigcup}
&\BDC_{\CC\mbox{\scriptsize -}c}(\CC_X).
\ar@{^{(}->}[u]_-{e_X\circ \iota_X}
}\]
\end{theorem}
Moreover, the author showed that there exists a t-structure
on the triangulated category $\BEC_{\CC\mbox{\scriptsize -}c}(\I\CC_X)$
whose heart is equivalent to the abelian category $\Modhol(\D_X)$
of holonomic $\D_X$-modules as follows.
We set
\begin{align*}
{}^p\bfE^{\leq0}_{\CC\mbox{\scriptsize -}c}(\I\CC_X) &:=
\{K\in\BEC_{\CC\mbox{\scriptsize -}c}(\I\CC_X)\ |\ \sh_X(K)\in{}^p\bfD^{\leq0}_{\CC\mbox{\scriptsize -}c}(\CC_X)\},\\
{}^p\bfE^{\geq0}_{\CC\mbox{\scriptsize -}c}(\I\CC_X) &:=
\{K\in\BEC_{\CC\mbox{\scriptsize -}c}(\I\CC_X)\ |\
\sh_X(K)\in{}^p\bfD^{\geq0}_{\CC\mbox{\scriptsize -}c}(\CC_X)\}\\
&\:=
\{K\in\BEC_{\CC\mbox{\scriptsize -}c}(\I\CC_X)\ |\
\rmD_X^{\rmE}(K)\in{}^p\bfE^{\leq0}_{\CC\mbox{\scriptsize -}c}(\I\CC_X)\},
\end{align*}
where 
$\sh_X\colon \BEC(\I\CC_X) \to \BDC(\CC_X)$ is the sheafification functor and
$\rmD_X^{\rmE}\colon \BEC(\I\CC_X)^\op\to \BEC(\I\CC_X)$ is the duality functor for enhanced ind-sheaves.
See \cite[Def.\:3.4]{DK21} for the definition of the functor $\sh$,
\cite[Def.\:4.8.1]{DK16} for the definition of the functor $\rmD_X^\rmE$.
Then we have:
\begin{theorem}[{\cite[Thm.\:4.5]{Ito20}}]\label{irregRH_ito}
The pair $\big({}^p\bfE^{\leq0}_{\CC\mbox{\scriptsize -}c}(\I\CC_X),
{}^p\bfE^{\geq0}_{\CC\mbox{\scriptsize -}c}(\I\CC_X)\big)$
is a t-structure on $\BEC_{\CC\mbox{\scriptsize -}c}(\I\CC_X)$
and
its heart $$\Perv(\I\CC_X) :=
{}^p\bfE^{\leq0}_{\CC\mbox{\scriptsize -}c}(\I\CC_X)\cap
{}^p\bfE^{\geq0}_{\CC\mbox{\scriptsize -}c}(\I\CC_X)$$
 is equivalent to the abelian category $\Modhol(\D_X)$:
$$\Sol_X^\rmE(\cdot)[d_X]\colon \Modhol(\D_X)^\op\simto \Perv(\I\CC_X).$$
Moreover, the following diagram is commutative:
\[\xymatrix@R=25pt@C=60pt@M=7pt{
\Modhol(\D_X)^{\op}\ar@{->}[r]^\sim\ar@<1.0ex>@{}[r]^-{\Sol_X^{\rmE}(\cdot)[d_X]}
 & \Perv(\I\CC_X)\\
\Modrh(\D_X)^{\op}\ar@{->}[r]_-{\Sol_X(\cdot)[d_X]}^-{\sim}\ar@{}[u]|-{\bigcup}
&\Perv(\CC_X).
\ar@{^{(}->}[u]_-{e_X\circ \iota_X}
}\]
\end{theorem}
Remark that the author reproved the regular Riemann--Hilbert correspondence (Theorem \ref{thm_regRH})
by using the irregular Riemann--Hilbert correspondence (Theorem \ref{irregRH_ana})
in \cite{Ito23}.

\subsection{Algebraic Irregular Riemann--Hilbert Correspondence and Enhanced Ind-Sheaves}
Moreover, the author proved an algebraic version of Theorem \ref{irregRH_ana} in \cite{Ito21} as below.
Let $X$ be a smooth algebraic variety over $\CC$
and denote by $\tl{X}$ a smooth completion of $X$
(i.e., $\tl{X}$ is a smooth complete algebraic variety over $\CC$
which contains $X$ as an open subvariety and $\tl{X}\setminus X$ is a normal crossing divisor of $\tl{X}$).
Remark that thanks to Nagata's compactification theorem \cite[Thm.\:4.3]{Naga}
and Hironaka's desingularization theorem \cite{Hiro},
such an $\widetilde{X}$ exists
for any smooth algebraic variety $X$ over $\CC$.

We denote by $\BEC_{\CC\mbox{\scriptsize -}c}(\I\CC_{X_\infty})$
the triangulated category of algebraic $\CC$-constructible enhanced ind-sheaves
on a bordered space $X_\infty^\an = (X^\an, \tl{X}^\an)$.
Here, $X^\an$ (resp.\:$\tl{X}^\an$) is the underlying complex manifold of $X$ (resp.\:$\tl{X}$).
See \cite[Def.\:3.2.1]{DK16} for the definition of bordered spaces,
\cite[Defs.\:3.1, 3.10]{Ito21} for the definition of algebraic $\CC$-constructible enhanced ind-sheaves.

\begin{theorem}[{\cite[Thm.\:3.11, Prop.\:3.14]{Ito21}}]\label{irregRH_alg}
There exists an equivalence of categories
\[\Sol_{X_\infty}^{\rmE} \colon \BDChol(\D_X)^{\op}\simto \BEC_{\CC\mbox{\scriptsize -}c}(\I\CC_{X_\infty})\]
and the following diagram is commutative:
\[\xymatrix@R=25pt@C=60pt@M=7pt{
\BDChol(\D_X)^{\op}\ar@{->}[r]^\sim\ar@<1.0ex>@{}[r]^-{\Sol_{X_\infty}^{\rmE}}
 & \BEC_{\CC\mbox{\scriptsize -}c}(\I\CC_{X_\infty})\\
\BDCrh(\D_X)^{\op}\ar@{->}[r]_-{\Sol_X}^-{\sim}\ar@{}[u]|-{\bigcup}
&\BDC_{\CC\mbox{\scriptsize -}c}(\CC_X),
\ar@{^{(}->}[u]_-{e_{X_\infty^\an}\circ \iota_{X_\infty^\an}}
}\]
where the second horizontal arrow is the algebraic regular Riemann--Hilbert correspondence.
\end{theorem}
Here, $e_{X_\infty^\an}\circ \iota_{X_\infty^\an}\colon \BDC(\CC_X) \to \BEC(\I\CC_{X_\infty^\an})$
is the natural embedding functor
from the derived category $\BDC(\CC_X)$ of sheaves of $\CC$-vector spaces on $X$ to
the triangulated category $\BEC(\I\CC_{X_\infty^\an})$ of enhanced ind-sheaves on $X_\infty^\an$.
See \cite[Def.\:2.19]{KS16} and also \cite[The paragraph before Notation 2.3]{DK21} for the details.
See  \cite[\S 3.2]{Ito21} for the definition of $\Sol_{X_\infty}^\rmE$.

Moreover, the author showed that there exists a t-structure
on the triangulated category $\BEC_{\CC\mbox{\scriptsize -}c}(\I\CC_{X_\infty})$
whose heart is equivalent to the abelian category $\Modhol(\D_X)$
of algebraic holonomic $\D_X$-modules as follows.
Let us recall that the triangulated category $\BDC_{\CC\mbox{\scriptsize -}c}(\CC_X)$
of algebraic $\CC$-constructible sheaves on $X^\an$ has the perverse t-structure 
$\big({}^p\bfD^{\leq0}_{\CC\mbox{\scriptsize -}c}(\CC_X),
{}^p\bfD^{\geq0}_{\CC\mbox{\scriptsize -}c}(\CC_X)\big)$
and denote by $\Perv(\CC_X)$ the heart of its t-structure,
see \cite{BBD} (also \cite[Thm.\:8.1.27]{HTT}) for the details.
We set
\begin{align*}
{}^p\bfE^{\leq0}_{\CC\mbox{\scriptsize -}c}(\I\CC_{X_\infty}) &:=
\{K\in\BEC_{\CC\mbox{\scriptsize -}c}(\I\CC_{X_\infty})\ |\
\sh_{X^\an_\infty}(K)\in{}^p\bfD^{\leq0}_{\CC\mbox{\scriptsize -}c}(\CC_X)\},\\
{}^p\bfE^{\geq0}_{\CC\mbox{\scriptsize -}c}(\I\CC_{X_\infty}) &:=
\{K\in\BEC_{\CC\mbox{\scriptsize -}c}(\I\CC_{X_\infty})\ |\ 
\sh_{X^\an_\infty}(K)\in{}^p\bfD^{\geq0}_{\CC\mbox{\scriptsize -}c}(\CC_X)\}\\
&\ =
\{K\in\BEC_{\CC\mbox{\scriptsize -}c}(\I\CC_{X_\infty})\ |\
\rmD_{X^\an_\infty}^{\rmE}(K)\in{}^p\bfE^{\leq0}_{\CC\mbox{\scriptsize -}c}(\I\CC_{X_\infty})\}
\end{align*}
and we have:

\begin{theorem}[{\cite[Thm.\:3.20, Prop.\:4.8]{Ito21}}]
The pair $\big({}^p\bfE^{\leq0}_{\CC\mbox{\scriptsize -}c}(\I\CC_{X_\infty}),
{}^p\bfE^{\geq0}_{\CC\mbox{\scriptsize -}c}(\I\CC_{X_\infty})\big)$
is a t-structure on $\BEC_{\CC\mbox{\scriptsize -}c}(\I\CC_{X_\infty})$
and
its heart $$\Perv(\I\CC_{X_\infty}) :=
{}^p\bfE^{\leq0}_{\CC\mbox{\scriptsize -}c}(\I\CC_{X_\infty})\cap{}^p\bfE^{\geq0}_{\CC\mbox{\scriptsize -}c}(\I\CC_{X_\infty})$$
is equivalent to the abelian category $\Modhol(\D_X)$ of holonomic $\D_X$-modules:
\[\Sol_{X_\infty}^\rmE(\cdot)[d_X] \colon \Modhol(\D_X)^{\op}\simto\Perv(\I\CC_{X_\infty}).\]
Moreover, the following diagram is commutative:
\[\xymatrix@R=25pt@C=60pt@M=7pt{
\Modhol(\D_X)^{\op}\ar@{->}[r]^\sim\ar@<1.0ex>@{}[r]^-{\Sol_{X_\infty}^{\rmE}(\cdot)[d_X]}
 & \Perv(\I\CC_{X_\infty})\\
\Modrh(\D_X)^{\op}\ar@{->}[r]_-{\Sol_X(\cdot)[d_X]}^-{\sim}\ar@{}[u]|-{\bigcup}
&\Perv(\CC_X).
\ar@{^{(}->}[u]_-{e_{X_\infty^\an}\circ \iota_{X_\infty^\an}}
}\]
\end{theorem}

\subsection{Irregular Riemann--Hilbert Correspondence and Enhanced Subanalytic Sheaves}
Let $X$ be a complex manifold again.
At the 16th Takagi lecture,
M.\:Kashiwara explained a similar result of Theorem\:\ref{irregRH_DK}
by using subanalytic sheaves instead of enhanced ind-sheaves as below.
We denote by $\BDC(\CC_{X\times\RR_\infty}^{\sub})$ the derived category of subanalytic sheaves
on an analytic bordered space $X\times\RR_\infty$. 
See \cite[\S\:3.4]{Kas16} and also \cite[\S\:3.1]{Ito24a}
for the definition of subanalytic sheaves on real analytic bordered spaces,
\cite[Def.\:5.4.1]{DK16} for the definition of real analytic bordered spaces.
See also \cite[\S\:6]{KS01} and \cite{Pre08}.
\begin{theorem}[{\cite[\S 5.4]{Kas16}}]
Then there exists a fully faithful functor: 
$$\Sol_X^{\T}\colon \BDChol(\D_X)^{\op}\hookrightarrow\BDC(\CC^\sub_{X\times \RR_\infty}).$$
\end{theorem}

In \cite{Ito24a},
the author defined enhanced subanalytic sheaves,
and explained a relation between enhanced ind-sheaves and enhanced subanalytic sheaves.
Let us denote by
$\BEC(\CC_X^\sub)$ the triangulated category of enhanced subanalytic sheaves on $X$
and denote by $\BEC_{\RR\mbox{\scriptsize -}c}(\CC_X^\sub)$
the one of $\RR$-constructible enhanced ind-sheaves on $X$.
See \cite[\S.\:3.3, Def.\:3.19]{Ito24a} for the definitions.
Then we have:

\begin{theorem}[{\cite[Thms.\:3.15, 3.20, 3.38, 3.39]{Ito24a}}]\label{thm1.8}
There exist fully faithful functors:
$$I^\rmE_X\colon \BEC(\CC_X^\sub)\hookrightarrow \BEC(\I\CC_X),\hspace{17pt}
\bfR_X^{\rmE, \sub}\colon \BEC(\CC_X^\sub)\hookrightarrow \BDC(\CC^\sub_{X\times \RR_\infty})$$
and the functor $I_X^\rmE$ induces an equivalence of categories:
$$I^\rmE_X\colon \BEC_{\RR\mbox{\scriptsize -}c}(\CC_X^\sub)\simto
\BEC_{\RR\mbox{\scriptsize -}c}(\I\CC_X).$$
Moreover, there exists a fully faithful functor
$$\Sol_X^{\rmE, \sub}\colon \BDChol(\D_X)^{\op}\hookrightarrow\BEC_{\RR\mbox{\scriptsize -}c}(\CC_X^\sub)$$
and the following diagram is commutative:
\[\xymatrix@M=7pt@R=25pt@C=60pt{
{}&{}&\BDC(\CC_{X\times\RR_\infty}^\sub)\\
\BDChol(\D_X)^\op
\ar@{^{(}->}[r]_-{\Sol_X^{\rmE, \sub}}
\ar@{^{(}->}[rru]
\ar@<0.6ex>@{}[rru]^-{\Sol_X^{\T, \sub}(\cdot)[1]}
\ar@{^{(}->}[rd]_-{\Sol_X^{\rmE}}
 & \BEC_{\RR\mbox{\scriptsize -}c}(\CC_X^\sub)\ar@{}[r]|-{\text{\large $\subset$}}
 \ar@<-1.0ex>@{->}[d]_-{I_X^\rmE}\ar@{}[d]|-\wr
  & \BEC(\CC_X^\sub)\ar@{^{(}->}[u]_-{\bfR_X^{\rmE, \sub}}
   \ar@<-1.0ex>@{^{(}->}[d]^-{I_X^\rmE}\\
{}&\BEC_{\RR\mbox{\scriptsize -}c}(\I\CC_X)\ar@{}[r]|-{\text{\large $\subset$}}
&\BEC(\I\CC_X).
}\]
\end{theorem}

Moreover, the author defined $\CC$-constructibility for enhanced subanalytic sheaves,
and prove that there exists an equivalence of categories  between
the triangulated category $\BEC_{\CC\mbox{\scriptsize -}c}(\CC_X^\sub)$
of $\CC$-constructible enhanced subanalytic sheaves on $X$
and that of holonomic $\D_X$-modules.
See \cite[Def.\:4.16]{Ito24b} for the definition of $\BEC_{\CC\mbox{\scriptsize -}c}(\CC_X^\sub)$.

\begin{theorem}[{\cite[Thms.\:4.19, 4.22, Prop.\:4.21]{Ito24b}}]\label{thm1.9}
The functor $I_X^\rmE$ induces an equivalence of categories:
$$I^\rmE_X\colon \BEC_{\CC\mbox{\scriptsize -}c}(\CC_X^\sub)\simto
\BEC_{\CC\mbox{\scriptsize -}c}(\I\CC_X).$$
Moreover, the functor $\Sol_{X}^{\rmE, \sub}$ induces an equivalence of categories:
\[\Sol_{X}^{\rmE, \sub} \colon \BDChol(\D_{X})^{\op}
\simto\BEC_{\CC\mbox{\scriptsize -}c}(\CC_{X}^\sub)\]
and the following diagram is commutative:
\[\xymatrix@R=25pt@C=60pt@M=7pt{
\BDChol(\D_X)^{\op}\ar@{->}[r]^\sim\ar@<1.0ex>@{}[r]^-{\Sol_X^{\rmE,\sub}}
 & \BEC_{\CC\mbox{\scriptsize -}c}(\CC_X^\sub)\\
\BDCrh(\D_X)^{\op}\ar@{->}[r]_-{\Sol_X}^-{\sim}\ar@{}[u]|-{\bigcup}
&\BDC_{\CC\mbox{\scriptsize -}c}(\CC_X).
\ar@{^{(}->}[u]_-{e_X^\sub\circ \rho_{X\ast}}
}\]
\end{theorem}
See \cite[Prop.\:4.21]{Ito24b} for the detail of the functor
$e_X^\sub\circ \rho_{X\ast}\colon
\BDC_{\CC\mbox{\scriptsize -}c}(\CC_X) \to
\BEC_{\CC\mbox{\scriptsize -}c}(\CC_X^\sub)$.

One can summarize the above results about the irregular Riemann--Hilbert correspondence
in the following commutative diagram:
\[\xymatrix@M=7pt@R=55pt@C=55pt{
{}&{}&{}&\BDC(\CC_{X\times\RR_\infty}^\sub) & {}\\
\BDChol(\D_X)^\op
\ar@{->}[r]_-{\Sol_X^{\rmE, \sub}}
  \ar@<-0.3ex>@{}[r]^-\sim
\ar@{^{(}->}[rrru]^-{\Sol_X^{\T, \sub}(\cdot)[1]}
\ar@{->}[rd]_-{\Sol_X^{\rmE}}
  \ar@<-1.0ex>@{}[rd]^-{\rotatebox{-25}{$\sim$}}
 & \BEC_{\CC\mbox{\scriptsize -}c}(\CC_X^\sub)
 \ar@{}[r]|-{\text{\large $\subset$}}
  \ar@<-1.0ex>@{->}[d]_-{I_X^\rmE}
 \ar@{}[d]|-\wr
 & \BEC_{\RR\mbox{\scriptsize -}c}(\CC_X^\sub)
 \ar@{}[r]|-{\text{\large $\subset$}}
 \ar@<-1.0ex>@{->}[d]_-{I_X^\rmE}
 \ar@{}[d]|-\wr
  & \BEC(\CC_X^\sub)
  \ar@{^{(}->}[u]_-{\bfR_X^{\rmE, \sub}}
  \ar@<-1.0ex>@{^{(}->}[d]_-{I_X^\rmE}\\
{}&\BEC_{\CC\mbox{\scriptsize -}c}(\I\CC_X)
\ar@{}[r]|-{\text{\large $\subset$}}
&\BEC_{\RR\mbox{\scriptsize -}c}(\I\CC_X)
\ar@{}[r]|-{\text{\large $\subset$}}
&\BEC(\I\CC_X).
}\]

\subsection{Algebraic Irregular Riemann--Hilbert Correspondence and Enhanced Subanalytic Sheaves}
In \cite{Ito24b}, the author proved an algebraic version of Theorem \ref{thm1.9}.
Let $X$ be a smooth algebraic variety over $\CC$ again,
and denote by $\tl{X}$ a smooth completion of $X$.
We denote by $\BEC_{\CC\mbox{\scriptsize -}c}(\CC_{X_\infty}^\sub)$
the triangulated category of algebraic $\CC$-constructible enhanced subanalytic sheaves
on a real analytic bordered space $X_\infty^\an = (X^\an, \tl{X}^\an)$.
See \cite[Defs.\:5.5, 5.9]{Ito24b} for the definition of algebraic $\CC$-constructible enhanced subanalytic sheaves.

\begin{theorem}[{\cite[Thm.\:5.6, 5.10, Prop.\:5.12]{Ito24b}}]
The functor $I_{X_\infty^\an}^\rmE$ induces an equivalence of categories:
$$I^\rmE_{X_\infty^\an}\colon \BEC_{\CC\mbox{\scriptsize -}c}(\CC_{X_\infty}^\sub)\simto
\BEC_{\CC\mbox{\scriptsize -}c}(\I\CC_{X_\infty}).$$
Moreover, there exists an equivalence of categories:
\[\Sol_{X_\infty}^{\rmE, \sub} \colon \BDChol(\D_{X})^{\op}
\simto\BEC_{\CC\mbox{\scriptsize -}c}(\CC_{X_\infty}^\sub)\]
and the following diagram is commutative:
\[\xymatrix@R=25pt@C=60pt@M=7pt{
\BDChol(\D_X)^{\op}\ar@{->}[r]^\sim\ar@<1.0ex>@{}[r]^-{\Sol_{X_\infty}^{\rmE,\sub}}
 & \BEC_{\CC\mbox{\scriptsize -}c}(\CC_{X_\infty}^\sub)\\
\BDCrh(\D_X)^{\op}\ar@{->}[r]_-{\Sol_X}^-{\sim}\ar@{}[u]|-{\bigcup}
&\BDC_{\CC\mbox{\scriptsize -}c}(\CC_X).
\ar@{^{(}->}[u]_-{e_{X_\infty^\an}^\sub\circ \rho_{X_\infty^\an\ast}}
}\]
\end{theorem}
See \cite[Prop.\:4.21]{Ito24b} for the detail of the functor
$e_{X_\infty^\an}^\sub\circ \rho_{X_\infty^\an\ast}\colon
\BDC_{\CC\mbox{\scriptsize -}c}(\CC_X) \to
\BEC_{\CC\mbox{\scriptsize -}c}(\CC_X^\sub)$.
See \cite[Prop.\:5.21]{Ito24b} for the details.

One can summarize the above results about the algebraic irregular Riemann--Hilbert correspondence
in the following commutative diagram:
\[\xymatrix@M=7pt@R=55pt@C=55pt{
\BDChol(\D_X)^\op
\ar@{->}[r]_-{\Sol_{X_\infty}^{\rmE, \sub}}
  \ar@<-0.3ex>@{}[r]^-\sim
\ar@{->}[rd]_-{\Sol_{X_\infty}^{\rmE}}
  \ar@<-1.0ex>@{}[rd]^-{\rotatebox{-25}{$\sim$}}
 & \BEC_{\CC\mbox{\scriptsize -}c}(\CC_{X_\infty}^\sub)
 \ar@{}[r]|-{\text{\large $\subset$}}
  \ar@<-1.0ex>@{->}[d]_-{I_{X_\infty^\an}^\rmE}
 \ar@{}[d]|-\wr
  & \BEC(\CC_{X_\infty^\an}^\sub)
  \ar@<-1.0ex>@{^{(}->}[d]_-{I_{X_\infty^\an}^\rmE}\\
{}&\BEC_{\CC\mbox{\scriptsize -}c}(\I\CC_{X_\infty})
\ar@{}[r]|-{\text{\large $\subset$}}
&\BEC(\I\CC_{X_\infty^\an}).
}\]

In this paper, 
we will show that there exists a t-structure
on the triangulated category of $\CC$-constructible enhanced subanalytic sheaves 
whose heart is equivalent to the abelian category of holonomic $\D$-modules.
Although it may be known by experts, it is not in the literature to our knowledge. 
The main definitions are Definition \ref{main-def-1} for the analytic case, Definition \ref{main-def-2} for the algebraic case,
and main results are Theorems \ref{main-thm-1}, \ref{main-thm-2} for the analytic case,
Theorems \ref{main-thm-3}, \ref{main-thm-4} for the algebraic case.
Furthermore, we shall consider simple objects of its heart and minimal extensions of objects of its heart
in Section 2.3.
The main definitions are Definition \ref{main-def-3}, \ref{main-def-4} and
the main results are Propositions \ref{main-thm-5}, \ref{main-thm-6}. 

\section*{Acknowledgement}
Since this paper was written in the wake of
``7th Tunisian--Japanese Conference, Geometric and Harmonic Analysis on Homogeneous Spaces and
Applications --- in Honor of Professor Toshiyuki Kobayashi ---",
I am grateful the organizers Professor Ali Baklouti, Professor Hideyuki Ishi and 
local organizers (especially, Manar Isram Toumi, Rahma Jerbi and so on ... ).

I would like to thank Dr.\:Taito Tauchi of Aoyama Gakuin University
for many discussions and giving many comments.
I would also like to thank Dr.\:Yuichiro Tanaka of the University of Tokyo,
Dr.\:Masatoshi Kitagawa of the University of Tokyo and
Associate Professor Yasufumi Nitta of Tokyo University of Science
for his continued support and constant encouragement. 

The author would like to thank administrative staff of Tokyo University of Science
(in particular Department of Mathematics, Faculty of Science Division II),
for his continued support and constant encouragement.

Finally, I am sincerely grateful for continued supports and constant encouragements of Professor Toshiyuki Kobayashi.
He always told me what is important.
Moreover I learned from Professor Kobayashi his attitude toward studying mathematics.
I cannot thank him enough.
I would like to return the favor to him by growing up someday.

\section*{Funding}
This work was supported by Grant-in-Aid for Early-Career Scientists (No.\:22K13902), 
Japan Society for the Promotion of Science.

\section{Main Result}

\subsection{Analytic Case}
The main definition of this section is Definition \ref{main-def-1}
and main results of this section are Theorems \ref{main-thm-1}, \ref{main-thm-2}.

Let $X$ be a complex manifold
and denote by $\BEC_{\CC\mbox{\scriptsize -}c}(\CC_X^\sub)$
the triangulated category of $\CC$-constructible enhanced subanalytic sheaves.
See \cite[The paragraph before Lem.\:3.10]{Ito24a} for the definition of enhanced subanalytic sheaves,
\cite[Def.\:4.16]{Ito24b} for the definition of $\CC$-constructible enhanced subanalytic sheaves.
Let us recall that there exist functors
$$\rho_X^{-1}\colon \BDC(\CC_X^\sub)\to \BDC(\CC_X),\hspace{17pt}
\sh_X^\sub\colon \BEC(\CC_X^\sub)\to \BDC(\CC_X^\sub)$$
and
$$\rho_X^{-1}\circ \sh_X^\sub\colon \BEC(\CC_X^\sub)\to \BDC(\CC_X).$$
See \cite[p.5]{Pre08} and also \cite[\S\:3.1]{Ito24a} for the definition of $\rho_X^{-1}$,
\cite[The paragraph before Prop.\:3.23]{Ito24a} for the definition of $\sh_X^\sub$.

\begin{definition}\label{main-def-1}
Let us define full subcategories of $\BEC_{\CC\mbox{\scriptsize -}c}(\CC_X^\sub)$ by
\begin{align*}
{}^p\bfE^{\leq0}_{\CC\mbox{\scriptsize -}c}(\CC_X^\sub) &:=
\{K\in\BEC_{\CC\mbox{\scriptsize -}c}(\CC_X^\sub)\ |\
\rho_X^{-1}\(\sh_X^\sub(K)\)\in{}^p\bfD^{\leq0}_{\CC\mbox{\scriptsize -}c}(\CC_X)\},\\
{}^p\bfE^{\geq0}_{\CC\mbox{\scriptsize -}c}(\CC_X^\sub) &:=
\{K\in\BEC_{\CC\mbox{\scriptsize -}c}(\CC_X^\sub)\ |\
\rho_X^{-1}\(\sh_X^\sub(K)\)\in{}^p\bfD^{\geq0}_{\CC\mbox{\scriptsize -}c}(\CC_X)\}.
\end{align*}
Here, the pair $\big({}^p\bfD^{\leq0}_{\CC\mbox{\scriptsize -}c}(\CC_X),
{}^p\bfD^{\geq0}_{\CC\mbox{\scriptsize -}c}(\CC_X)\big)$
is the perverse t-structure on the triangulated category $\BDC_{\CC\mbox{\scriptsize -}c}(\CC_X)$
of $\CC$-constructible sheaves on $X$.
\end{definition}

Then there exists a relation between the pair
$\big({}^p\bfE^{\leq0}_{\CC\mbox{\scriptsize -}c}(\CC_X^\sub),
{}^p\bfE^{\geq0}_{\CC\mbox{\scriptsize -}c}(\CC_X^\sub)\big)$
and the t-structure  
$\big({}^p\bfE^{\leq0}_{\CC\mbox{\scriptsize -}c}(\I\CC_X),
{}^p\bfE^{\geq0}_{\CC\mbox{\scriptsize -}c}(\I\CC_X)\big)$
on $\BEC_{\CC\mbox{\scriptsize -}c}(\I\CC_X)$ which is defined in \cite[Def.\:4.2]{Ito20}.

\begin{proposition}\label{prop2.2}
For any $K\in \BEC_{\CC\mbox{\scriptsize -}c}(\CC_X^\sub)$,
we have
\begin{itemize}
\item[{\rm (1)}]
$K\in {}^p\bfE^{\leq0}_{\CC\mbox{\scriptsize -}c}(\CC_X^\sub)\
\Longleftrightarrow I_X^{\rmE}(K)\in {}^p\bfE^{\leq0}_{\CC\mbox{\scriptsize -}c}(\I\CC_X)$,

\item[{\rm (2)}]
$K\in {}^p\bfE^{\geq0}_{\CC\mbox{\scriptsize -}c}(\CC_X^\sub)\
\Longleftrightarrow I_X^{\rmE}(K)\in {}^p\bfE^{\geq0}_{\CC\mbox{\scriptsize -}c}(\I\CC_X)$.
\end{itemize}
\end{proposition}

\begin{proof}
This follows from Lemma \ref{lem2.3} below.
\end{proof}

\begin{lemma}\label{lem2.3}
For any $K\in\BEC(\CC_X^\sub)$,
there exists an isomorphism in $\BDC(\CC_X)$:
$$\rho_X^{-1}\(\sh_X^\sub(K)\)
\simeq \sh_X\left(I_X^\rmE(K)\right).$$
Namely, we have the following commutative diagram: 
\[\xymatrix@M=7pt@C=45pt{
\BEC(\CC_X^\sub)\ar@{->}[r]^-{I_X^\rmE}\ar@{->}[d]_-{\sh_X^\sub}
&
\BEC(\I\CC_X)\ar@{->}[d]^-{\sh_X}\\
\BDC(\CC_X^\sub)\ar@{->}[r]_-{\rho_X^{-1}}
&
\BDC(\CC_X).}
\]
\end{lemma}

\begin{proof}
Let us recall that
there exists a fully faithful functor $$I_X\colon \BDC(\CC_X^\sub)\hookrightarrow\BDC(\I\CC_X)$$
from the triangulated category $\BDC(\CC_X^\sub)$ of subanalytic sheaves on $X$ to
the triangulated category $\BDC(\I\CC_X)$ of ind-sheaves on $X$
which has the right adjoint functor
$$\bfR J_X\colon \BDC(\I\CC_X) \to \BDC(\CC_X^\sub).$$
See \cite[Prop.\:3.6]{Ito24a} for the details.
We also recall that there exists a functor $$\Ish_X\colon \BEC(\I\CC_X)\to \BDC(\I\CC_X)$$
such that $\sh_X = \alpha_X\circ \Ish_X.$
Here, $$\alpha_X\colon \BDC(\I\CC_X)\to \BDC(\CC_X)$$ is
an exact left adjoint of the fully faithful functor $$\iota_X\colon \BDC(\CC_X)\hookrightarrow \BDC(\I\CC_X).$$
See \cite[p.35]{DK21} for the definition of $\Ish_X$,
\cite[Def.\:3.3.18]{KS01} for the definition of $\alpha_X$.

Let $K\in\BEC(\CC_X^\sub)$.
Then we have isomorphisms in $\BDC(\CC_X)$:
\begin{align*}
\sh_X\(I_X^\rmE(K)\) &\simeq \alpha_X\(\Ish_X\(I_X^\rmE(K)\)\)\\
&\simeq
\rho_X^{-1}\(\bfR J_X\(\Ish_X\(I_X^\rmE(K)\)\)\)\\
&\simeq
\rho_X^{-1}\(\sh_X^{\sub}\(J_X^\rmE\(I_X^\rmE(K)\)\)\)\\
&\simeq
\rho_X^{-1}\(\sh_X^{\sub}(K)\),
\end{align*}
where in the second isomorphism we used \cite[Prop.3.7(3)(ii)]{Ito24a},
in the third isomorphism we used \cite[Prop.3.26(3)]{Ito24a}
and in the forth isomorphism we used \cite[Thm.3.18(1)]{Ito24a}.

The proof is completed.
\end{proof}

Then we obtain the following two theorems which are the main theorems of this paper.
\begin{theorem}\label{main-thm-1}
The pair $\big({}^p\bfE^{\leq0}_{\CC\mbox{\scriptsize -}c}(\CC_X^\sub),
{}^p\bfE^{\geq0}_{\CC\mbox{\scriptsize -}c}(\CC_X^\sub)\big)$
is a t-structure on $\BEC_{\CC\mbox{\scriptsize -}c}(\CC_X^\sub)$
and its heart
$$\Perv(\CC_X^\sub) :=
{}^p\bfE^{\leq0}_{\CC\mbox{\scriptsize -}c}(\CC_X^\sub)\cap
{}^p\bfE^{\geq0}_{\CC\mbox{\scriptsize -}c}(\CC_X^\sub)$$
is equivalent to the abelian category $\Perv(\I\CC_X)$ which is defined in \cite[Thm.\:4.5]{Ito20}{\rm :}
$$I_X^{\rmE}\colon \Perv(\CC_X^\sub)\simto \Perv(\I\CC_X).$$
\end{theorem}

\begin{proof}
The first assertion follows from Proposition \ref{prop2.2}, \cite[Thm.\:4.5]{Ito20} and \cite[Thm.\:4.19]{Ito24b}.
The second assertion follows from Proposition \ref{prop2.2} and \cite[Thm.\:4.19]{Ito24b}.
\end{proof}

Let us denote by $\(\bfD^{\leq0}_{\rm hol}(\D_X), \bfD^{\leq0}_{\rm hol}(\D_X)\)$
the standard t-structure on $\BDChol(\D_X)$
such that
$\Modhol(\D_X) = \bfD^{\leq0}_{\rm hol}(\D_X)\cap \bfD^{\leq0}_{\rm hol}(\D_X)$.
\begin{theorem}\label{main-thm-2}
For any $\M\in\BDChol(\D_X)$,
we have
\begin{itemize}
\item[{\rm (1)}]
$\M\in \bfD^{\leq0}_{\rm hol}(\D_X)\
\Leftrightarrow \Sol_X^{\rmE, \sub}(\M)\in {}^p\bfE^{\geq0}_{\CC\mbox{\scriptsize -}c}(\CC_X^\sub)$,

\item[{\rm (2)}]
$\M\in \bfD^{\leq0}_{\rm hol}(\D_X)\
\Leftrightarrow \Sol_X^{\rmE, \sub}(\M)\in {}^p\bfE^{\leq0}_{\CC\mbox{\scriptsize -}c}(\CC_X^\sub)$.
\end{itemize}
Moreover, the functor $\Sol_X^{\rmE, \sub}$ induces an equivalence of categories:
$$\Sol_X^{\rmE, \sub}(\cdot)[d_X]\colon \Modhol(\D_X)^\op\simto \Perv(\CC_X^\sub).$$
\end{theorem}

\begin{proof}
The assertions (1) and (2) follow from Proposition \ref{prop2.2},
\cite[Thm.\:4.5]{Ito20} and \cite[Thm.\:3.39]{Ito24a}.
The last assertion follows from (1), (2) and \cite[Thm.\:4.19]{Ito24b}.
\end{proof}

We shall call an object of $\Perv(\CC_X^\sub)$ an enhanced perverse subanalytic sheaf.
From now on we prove some properties of enhanced perverse subanalytic sheaves.

Let us denote by $\rmD_X^{\rmE, \sub}\colon \BEC(\CC_X^\sub)^\op\to\BEC(\CC_X^\sub)$
the duality functor for enhanced subanalytic sheaves.
See \cite[The paragraph before Lem.\:3.24]{Ito24a} for the definition.
\begin{proposition}
The duality functor $\rmD_X^{\rmE, \sub}$ induces an equivalence of categories:
$$\rmD_X^{\rmE,\sub} : \Perv(\CC_X^\sub)^{\op}\simto\Perv(\CC_X^\sub).$$
\end{proposition}
\begin{proof}
This proposition follows from Lemma \ref{lem2.7} below.
\end{proof}

\begin{lemma}\label{lem2.7}
We have
$${}^p\bfE^{\geq0}_{\CC\mbox{\scriptsize -}c}(\CC_X^\sub) =
\{K\in\BEC_{\CC\mbox{\scriptsize -}c}(\CC_X^\sub)\ |\
\rmD_X^{\rmE, \sub}(K)\in{}^p\bfE^{\leq0}_{\CC\mbox{\scriptsize -}c}(\CC_X^\sub)\}.$$
\end{lemma}

\begin{proof}
Let $K\in \BEC_{\CC\mbox{\scriptsize -}c}(\CC_X^\sub)$.
Then we have
$$K\in {}^p\bfE^{\geq0}_{\CC\mbox{\scriptsize -}c}(\CC_X^\sub)\
\Longleftrightarrow
I_X^{\rmE}(K)\in {}^p\bfE^{\geq0}_{\CC\mbox{\scriptsize -}c}(\I\CC_X)$$
by Proposition \ref{prop2.2} (2)
and we have 
$$I_X^{\rmE}(K)\in {}^p\bfE^{\geq0}_{\CC\mbox{\scriptsize -}c}(\I\CC_X)
\Longleftrightarrow
\rmD_X^\rmE\(I_X^{\rmE}(K)\)\in {}^p\bfE^{\leq0}_{\CC\mbox{\scriptsize -}c}(\I\CC_X)$$
by \cite[Lem.\:4.1]{Ito20}.
Here, $\rmD_X^{\rmE}\colon \BEC(\I\CC_X)^\op\to \BEC(\I\CC_X)$ is the duality functor for enhanced ind-sheaves.
See \cite[Def.\:4.8.1]{DK16} for the definition.
Since there exists an isomorphism in $\BDC(\CC_X^\sub)$
$$J_X^\rmE\(\rmD_X^\rmE\(I_X^{\rmE}(K)\)\)\simeq \rmD_X^{\rmE,\sub}(K)$$
by \cite[Prop.\:3.25 (1)]{Ito24a},
we have isomorphisms in $\BEC(\I\CC_X)$
$$I_X^\rmE\(\rmD_X^{\rmE,\sub}(K)\)
\simeq
I_X^\rmE\(J_X^\rmE\(\rmD_X^\rmE\(I_X^{\rmE}(K)\)\)\)
\simeq
\rmD_X^\rmE\(I_X^{\rmE}(K)\),$$
where in the second isomorphism
we used the fact $\rmD_X^\rmE\(I_X^{\rmE}(K)\)\in \BEC_{\CC\mbox{\scriptsize -}c}(\I\CC_X)$
and \cite[Thm.\:3.19]{Ito24b}.
Therefore we have
$$\rmD_X^\rmE\(I_X^{\rmE}(K)\)\in {}^p\bfE^{\leq0}_{\CC\mbox{\scriptsize -}c}(\I\CC_X)
\Longleftrightarrow
I_X^\rmE\(\rmD_X^{\rmE,\sub}(K)\) \in {}^p\bfE^{\leq0}_{\CC\mbox{\scriptsize -}c}(\I\CC_X)$$
and hence
$$I_X^\rmE\(\rmD_X^{\rmE,\sub}(K)\) \in {}^p\bfE^{\leq0}_{\CC\mbox{\scriptsize -}c}(\I\CC_X)
\Longleftrightarrow
\rmD_X^{\rmE,\sub}(K) \in {}^p\bfE^{\leq0}_{\CC\mbox{\scriptsize -}c}(\CC_X^\sub)$$
by Proposition \ref{prop2.2} (1).

The proof is completed.
\end{proof}

Recall that there exists the natural embedding functor:
$$e_X^\sub\circ \rho_{X\ast}\colon \BDC_{\CC\mbox{\scriptsize -}c}(\CC_X)
\hookrightarrow \BEC_{\CC\mbox{\scriptsize -}c}(\CC_X^\sub).$$
See \cite[Prop.\:4.21]{Ito24b} for the details.

\begin{proposition}
We have 
\begin{itemize}
\item[\rm (1)]
for any $\SF\in{}^p\bfD_{\CC\mbox{\scriptsize -}c}^{\leq0}(\CC_X)$,
we have $e_X^\sub\(\rho_{X\ast}(\SF)\)\in{}^p\bfE_{\CC\mbox{\scriptsize -}c}^{\leq0}(\CC_X^\sub)$,

\item[\rm (2)]
for any $\SF\in{}^p\bfD_{\CC\mbox{\scriptsize -}c}^{\geq0}(\CC_X)$,
we have $e_X^\sub\(\rho_{X\ast}(\SF)\)\in{}^p\bfE_{\CC\mbox{\scriptsize -}c}^{\geq0}(\CC_X^\sub)$.
\end{itemize}
Moreover, the functor $e_X^\sub\circ \rho_{X\ast}$ induces a fully faithful embedding:
$$e_X^\sub\circ\rho_{X\ast}\colon \Perv(\CC_X)\hookrightarrow \Perv(\CC_X^\sub)$$
and the following diagram is commutative:
\[\xymatrix@R=25pt@C=60pt@M=7pt{
\Modhol(\D_X)^{\op}\ar@{->}[r]^\sim\ar@<1.0ex>@{}[r]^-{\Sol_X^{\rmE, \sub}(\cdot)[d_X]}
 & \Perv(\CC_X^\sub)\\
\Modrh(\D_X)^{\op}\ar@{->}[r]_-{\Sol_X(\cdot)[d_X]}^-{\sim}\ar@{}[u]|-{\bigcup}
&\Perv(\CC_X).
\ar@{^{(}->}[u]_-{e_X\circ \rho_{X \ast}}
}\]
\end{proposition}

\begin{proof}
The first and second assertions follow from
\cite[Prop.\:4.8 (1), (2)]{Ito20}, \cite[Props.\:3.4 (4)(i), 3.21]{Ito24a} and Proposition \ref{prop2.2}.
The last assertion follows from the first and second assertions and \cite[Cor.\:4.23]{Ito24b}.
\end{proof}

\begin{proposition}\label{prop2.9}
Let $f : X\to Y$ be a proper morphism of complex manifolds.
We assume that there exists a non-negative integer $d\in \ZZ_{\geq0}$
such that $\dim f^{-1}(y)\leq d$ for any $y\in Y$.
Here, $\dim f^{-1}(y)$ is the dimension of $f^{-1}(y)$ as an analytic space. 
Then we have
\begin{itemize}
\item[\rm(1)]
for any $K\in{}^p\bfE_{\CC\mbox{\scriptsize -}c}^{\leq0}(\CC_X^\sub)$,
we have $\bfE f_{!!} K\in{}^p\bfE_{\CC\mbox{\scriptsize -}c}^{\leq d}(\CC_Y^\sub)$,
\item[\rm(2)]
for any $K\in{}^p\bfE_{\CC\mbox{\scriptsize -}c}^{\geq0}(\CC_X^\sub)$,
we have $\bfE f_{!!} K\in{}^p\bfE_{\CC\mbox{\scriptsize -}c}^{\geq -d}(\CC_Y^\sub)$.
\end{itemize}
\end{proposition}

\begin{proof}
The assertion (1) follows from Proposition \ref{prop2.2} (1),
\cite[Prop.\:4.9 (1)]{Ito20} and \cite[Prop.\:3.16 (2)(ii)]{Ito24a}.
The assertion (2) follows from Proposition \ref{prop2.2} (2),
\cite[Prop.\:4.9 (2)]{Ito20} and \cite[Prop.\:3.16 (2)(ii)]{Ito24a}.
\end{proof}

\begin{proposition}
Let $f : X\to Y$ be a morphism of complex manifolds.
We assume that there exists a non-negative integer $d\in \ZZ_{\geq0}$
such that $\dim f^{-1}(y)\leq d$ for any $y\in Y$.
Then we have
\begin{itemize}
\item[\rm(1)]
for any $L\in{}^p\bfE_{\CC\mbox{\scriptsize -}c}^{\leq0}(\CC_Y^\sub)$
we have $\bfE f^{-1} L\in{}^p\bfE_{\CC\mbox{\scriptsize -}c}^{\leq d}(\CC_X^\sub)$,
\item[\rm(2)]
for any $L\in{}^p\bfE_{\CC\mbox{\scriptsize -}c}^{\geq0}(\CC_Y^\sub)$
we have $\bfE f^! L\in{}^p\bfE_{\CC\mbox{\scriptsize -}c}^{\geq -d}(\CC_X^\sub)$.
\end{itemize}
\end{proposition}

\begin{proof}
The assertion (1) follows from Proposition \ref{prop2.2} (1),
\cite[Prop.\:4.10 (1)]{Ito20} and \cite[Prop.\:3.16 (2)(i)]{Ito24a}.
The assertion (2) follows from Proposition \ref{prop2.2} (2),
\cite[Prop.\:4.10 (2)]{Ito20} and \cite[Prop.\:3.16 (2)(iii)]{Ito24a}.
\end{proof}

\subsection{Algebraic Case}
The main definition of this section is Definition \ref{main-def-2}
and main results of this section are Theorems \ref{main-thm-3}, \ref{main-thm-4}.

Let $X$ be a smooth algebraic variety over $\CC$
and denote by $\tl{X}$ a smooth completion of $X$.
We shall recall that there exist functors
$$\rho_{X_\infty^\an}^{-1}\colon \BDC(\CC_{X_\infty^\an}^\sub)\to \BDC(\CC_X),\hspace{17pt}
\sh_{X_\infty^\an}^\sub\colon \BEC(\CC_{X_\infty^\an}^\sub)\to \BDC(\CC_{X_\infty^\an}^\sub)$$
and $$\rho_{X_\infty^\an}^{-1}\circ \sh_{X_\infty^\an}^\sub\colon
\BEC(\CC_{X_\infty^\an}^\sub)\to \BDC(\CC_X).$$
See \cite[\S\:3.1]{Ito24a} for the definition of $\rho_{X_\infty^\an}^{-1}$,
\cite[The paragraph before Prop.\:3.23]{Ito24a} for the definition of $\sh_{X_\infty^\an}^\sub$.

\begin{definition}\label{main-def-2}
We define full subcategories of $\BEC_{\CC\mbox{\scriptsize -}c}(\CC_{X_\infty}^\sub)$ by
\begin{align*}
{}^p\bfE^{\leq0}_{\CC\mbox{\scriptsize -}c}(\CC_{X_\infty}^\sub) &:=
\{K\in\BEC_{\CC\mbox{\scriptsize -}c}(\CC_{X_\infty}^\sub)\ |\
\rho_{X_\infty^\an}^{-1}\(\sh_{X^\an_\infty}(K)\)\in{}^p\bfD^{\leq0}_{\CC\mbox{\scriptsize -}c}(\CC_X)\},\\
{}^p\bfE^{\geq0}_{\CC\mbox{\scriptsize -}c}(\CC_{X_\infty}^\sub) &:=
\{K\in\BEC_{\CC\mbox{\scriptsize -}c}(\CC_{X_\infty}^\sub)\ |\ 
\rho_{X_\infty^\an}^{-1}\(\sh_{X^\an_\infty}(K)\)\in{}^p\bfD^{\geq0}_{\CC\mbox{\scriptsize -}c}(\CC_X)\}.
\end{align*}
Here, the pair $\big({}^p\bfD^{\leq0}_{\CC\mbox{\scriptsize -}c}(\CC_X),
{}^p\bfD^{\geq0}_{\CC\mbox{\scriptsize -}c}(\CC_X)\big)$
is the perverse t-structure on the triangulated category $\BDC_{\CC\mbox{\scriptsize -}c}(\CC_{X^\an})$
of algebraic $\CC$-constructible sheaves on $X^\an$.
See \cite[Def.\:4.5.6, Thm.\:8.1.27]{HTT} for the details.
\end{definition}

Then there exists a relation between the pair
$\big({}^p\bfE^{\leq0}_{\CC\mbox{\scriptsize -}c}(\CC_{X_\infty}^\sub),
{}^p\bfE^{\geq0}_{\CC\mbox{\scriptsize -}c}(\CC_{X_\infty}^\sub)\big)$
and the t-structure  
$\big({}^p\bfE^{\leq0}_{\CC\mbox{\scriptsize -}c}(\I\CC_{X_\infty}),
{}^p\bfE^{\geq0}_{\CC\mbox{\scriptsize -}c}(\I\CC_{X_\infty})\big)$
on $\BEC_{\CC\mbox{\scriptsize -}c}(\I\CC_{X_\infty})$ which is defined in \cite[Def.\:3.18]{Ito21}.
Let us recall that there exists an equivalence of categories:
$$I_{X_\infty^\an}^{\rmE}\colon
\BEC_{\CC\mbox{\scriptsize -}c}(\CC_{X_\infty}^\sub)\hookrightarrow
\BEC_{\CC\mbox{\scriptsize -}c}(\I\CC_{X_\infty}).$$
See \cite[Thm.\:5.10]{Ito24b} and also \cite[Thm.\:3.15]{Ito24a} for the details.

\begin{proposition}\label{prop2.12}
For any $K\in \BEC_{\CC\mbox{\scriptsize -}c}(\CC_{X_\infty}^\sub)$,
we have
\begin{itemize}
\item[{\rm (1)}]
$K\in {}^p\bfE^{\leq0}_{\CC\mbox{\scriptsize -}c}(\CC_{X_\infty}^\sub)\
\Longleftrightarrow I_{X_\infty^\an}^{\rmE}(K)\in {}^p\bfE^{\leq0}_{\CC\mbox{\scriptsize -}c}(\I\CC_{X_\infty})$,

\item[{\rm (2)}]
$K\in {}^p\bfE^{\geq0}_{\CC\mbox{\scriptsize -}c}(\CC_{X_\infty}^\sub)\
\Longleftrightarrow I_{X_\infty^\an}^{\rmE}(K)\in {}^p\bfE^{\geq0}_{\CC\mbox{\scriptsize -}c}(\I\CC_{X_\infty})$.
\end{itemize}
\end{proposition}

\begin{proof}
This follows from Lemma \ref{lem2.13} below.
\end{proof}

The following lemma can be proved in a similar way to Lemma \ref{lem2.3}.
We shall skip the proof of this lemma.
\begin{lemma}\label{lem2.13}
For any $K\in\BEC(\CC_{X_\infty}^\sub)$,
there exists an isomorphism in $\BDC(\CC_X)$:
$$\rho_{X_\infty^\an}^{-1}\(\sh_{X_\infty^\an}^\sub(K)\)
\simeq \sh_{X_\infty^\an}\left(I_{X_\infty^\an}^\rmE(K)\right).$$
\end{lemma}
Then we obtain the following two theorems which are main theorems of this paper.
\begin{theorem}\label{main-thm-3}
The pair $\big({}^p\bfE^{\leq0}_{\CC\mbox{\scriptsize -}c}(\CC_{X_\infty}^\sub),
{}^p\bfE^{\geq0}_{\CC\mbox{\scriptsize -}c}(\CC_{X_\infty}^\sub)\big)$
is a t-structure on $\BEC_{\CC\mbox{\scriptsize -}c}(\CC_{X_\infty}^\sub)$
and its heat
$$\Perv(\CC_{X_\infty}^\sub) :=
{}^p\bfE^{\leq0}_{\CC\mbox{\scriptsize -}c}(\CC_{X_\infty}^\sub)\cap
{}^p\bfE^{\geq0}_{\CC\mbox{\scriptsize -}c}(\CC_{X_\infty}^\sub)$$
is equivalent to the abelian category $\Perv(\I\CC_{X_\infty})$ which is defined in \cite[Def.\:3.20]{Ito21}{\rm :}
$$I_{X_\infty^\an}^{\rmE}\colon \Perv(\CC_{X_\infty}^\sub)\simto \Perv(\I\CC_{X_\infty}).$$
\end{theorem}

\begin{proof}
The first assertion follows from Proposition \ref{prop2.12}, \cite[Thm.\:3.19]{Ito21} and \cite[Thm.\:5.10]{Ito24b}.
The second assertion follows from Proposition \ref{prop2.12} and \cite[Thm.\:5.10]{Ito24b}.
\end{proof}

Let us denote by $\(\bfD^{\leq0}_{\rm hol}(\D_X), \bfD^{\leq0}_{\rm hol}(\D_X)\)$
the standard t-structure on $\BDChol(\D_X)$
such that
$\Modhol(\D_X) = \bfD^{\leq0}_{\rm hol}(\D_X)\cap \bfD^{\leq0}_{\rm hol}(\D_X)$.
\begin{theorem}\label{main-thm-4}
For any $\M\in\BDChol(\D_X)$,
we have
\begin{itemize}
\item[{\rm (1)}]
$\M\in \bfD^{\leq0}_{\rm hol}(\D_X)\
\Leftrightarrow
\Sol_{X_\infty}^{\rmE, \sub}(\M)\in {}^p\bfE^{\geq0}_{\CC\mbox{\scriptsize -}c}(\CC_{X_\infty}^\sub)$,

\item[{\rm (2)}]
$\M\in \bfD^{\leq0}_{\rm hol}(\D_X)\
\Leftrightarrow
\Sol_{X_\infty}^{\rmE, \sub}(\M)\in {}^p\bfE^{\leq0}_{\CC\mbox{\scriptsize -}c}(\CC_{X_\infty}^\sub)$.
\end{itemize}
Moreover, the functor $\Sol_{X_\infty}^{\rmE, \sub}$ induces an equivalence of categories:
$$\Sol_{X_\infty}^{\rmE, \sub}(\cdot)[d_X]\colon \Modhol(\D_X)^\op\simto \Perv(\CC_{X_\infty}^\sub).$$
\end{theorem}

\begin{proof}
The assertions (1) and (2) follow from Proposition \ref{prop2.12},
\cite[Thm.\:3.19]{Ito21} and \cite[Thm.\:5.14]{Ito24b}.
The last assertion follows from (1), (2) and \cite[Thm.\:5.14]{Ito24b}.
\end{proof}

We shall call an object of $\Perv(\CC_{X_\infty}^\sub)$ an algebraic enhanced perverse subanalytic sheaf.
From now on we prove some properties of algebraic enhanced perverse subanalytic sheaves.

Let us denote by
$\rmD_{X_\infty^\an}^{\rmE, \sub}\colon \BEC(\CC_{X_\infty^\an}^\sub)^\op\to\BEC(\CC_{X_\infty^\an}^\sub)$
the duality functor for enhanced subanalytic sheaves.
See \cite[The paragraph before Lem.\:3.24]{Ito24a} for the definition.
\begin{proposition}
The duality functor $\rmD_{X_\infty^\an}^{\rmE, \sub}$ induces an equivalence of categories:
$$\rmD_{X_\infty^\an}^{\rmE,\sub} : \Perv(\CC_{X_\infty}^\sub)^{\op}\simto\Perv(\CC_{X_\infty}^\sub).$$
\end{proposition}
\begin{proof}
This proposition follows from Lemma \ref{lem2.17} below.
\end{proof}

The following lemma can be proved in a similar way to Lemma \ref{lem2.7}.
We shall skip the proof of this lemma.
\begin{lemma}\label{lem2.17}
We have
$${}^p\bfE^{\geq0}_{\CC\mbox{\scriptsize -}c}(\CC_{X_\infty}^\sub) =
\{K\in\BEC_{\CC\mbox{\scriptsize -}c}(\CC_{X_\infty}^\sub)\ |\
\rmD_{X_\infty^\an}^{\rmE, \sub}(K)\in{}^p\bfE^{\leq0}_{\CC\mbox{\scriptsize -}c}(\CC_{X_\infty}^\sub)\}.$$
\end{lemma}

Recall that there exists the natural embedding functor:
$$e_{X_\infty^\an}^\sub\circ \rho_{{X_\infty^\an}\ast}\colon \BDC_{\CC\mbox{\scriptsize -}c}(\CC_X)
\hookrightarrow \BEC_{\CC\mbox{\scriptsize -}c}(\CC_X^\sub).$$
See \cite[Prop.\:4.21]{Ito24b} for the details.

\begin{proposition}
We have 
\begin{itemize}
\item[\rm (1)]
for any $\SF\in{}^p\bfD_{\CC\mbox{\scriptsize -}c}^{\leq0}(\CC_X)$,
we have $e_{X_\infty^\an}^\sub\(\rho_{{X_\infty^\an}\ast}(\SF)\)\in
{}^p\bfE_{\CC\mbox{\scriptsize -}c}^{\leq0}(\CC_{X_\infty}^\sub)$,

\item[\rm (2)]
for any $\SF\in{}^p\bfD_{\CC\mbox{\scriptsize -}c}^{\geq0}(\CC_X)$,
we have $e_{X_\infty^\an}^\sub\(\rho_{{X_\infty^\an}\ast}(\SF)\)\in
{}^p\bfE_{\CC\mbox{\scriptsize -}c}^{\geq0}(\CC_{X_\infty}^\sub)$.
\end{itemize}
Moreover, the functor $e_{X_\infty^\an}^\sub\circ \rho_{{X_\infty^\an}\ast}$ induces a fully faithful embedding:
$$e_{X_\infty^\an}^\sub\circ\rho_{{X_\infty^\an}\ast}\colon
\Perv(\CC_X)\hookrightarrow \Perv(\CC_{X_\infty}^\sub)$$
and the following diagram is commutative:
\[\xymatrix@R=25pt@C=60pt@M=7pt{
\Modhol(\D_X)^{\op}\ar@{->}[r]^\sim\ar@<1.0ex>@{}[r]^-{\Sol_{X_\infty}^{\rmE, \sub}(\cdot)[d_X]}
 & \Perv(\CC_{X_\infty}^\sub)\\
\Modrh(\D_X)^{\op}\ar@{->}[r]_-{\Sol_X(\cdot)[d_X]}^-{\sim}\ar@{}[u]|-{\bigcup}
&\Perv(\CC_X).
\ar@{^{(}->}[u]_-{e_{X_\infty^\an}^\sub\circ \rho_{{X_\infty^\an} \ast}}
}\]
\end{proposition}

\begin{proof}
The first, second and third assertions follow from
\cite[Prop.\:3.21]{Ito21}, \cite[Props.\:3.4 (4)(i), 3.21]{Ito24a} and Proposition \ref{prop2.12}.
The last assertion follows from the first and second assertions and \cite[Thm.\:5.14]{Ito24b}.
\end{proof}

Let us recall that there exist a morphism $j_{X_\infty^\an}\colon X_\infty^\an\to \tl{X}^\an$ of bordered spaces 
and a functor
$$\bfE j_{X_\infty^\an!!}\colon \BEC(\CC_{X_\infty^\an}^\sub)\to \BEC(\CC_{\tl{X}^\an}^\sub).$$
See \cite[\S\S\:2.4, 3.3]{Ito24a} for the details.
Then we have the following lemma.
Remark that a bordered space $\tl{X}_\infty^\an$ is isomorphic to $\tl{X}^\an$,
where $\tl{X}^\an$ is identified with a bordered space $(\tl{X}^\an, \tl{X}^\an)$. 
See \cite[Lem.\:2.3]{Ito21} for the details.
Hence, we can define a t-structure 
$\({}^p\bfE_{\CC\mbox{\scriptsize -}c}^{\leq0}(\CC_{\tl{X}}^\sub),
{}^p\bfE_{\CC\mbox{\scriptsize -}c}^{\geq0}(\CC_{\tl{X}}^\sub)\)$
on $\BEC_{\CC\mbox{\scriptsize -}c}(\CC_{\tl{X}}^\sub)$
by Definition \ref{main-def-2}.

\begin{lemma}\label{lem2.19}
For any $K\in\BEC_{\CC\mbox{\scriptsize -}c}(\I\CC_{X_\infty})$,
we have
\begin{itemize}
\item[\rm(1)]
$K\in{}^p\bfE_{\CC\mbox{\scriptsize -}c}^{\leq0}(\CC_{X_\infty}^\sub)
\Longleftrightarrow
\bfE j_{X^\an_\infty!!}(K)\in{}^p\bfE_{\CC\mbox{\scriptsize -}c}^{\leq0}(\CC_{\tl{X}}^\sub)$,
\item[\rm(2)]
$K\in{}^p\bfE_{\CC\mbox{\scriptsize -}c}^{\geq0}(\I\CC_{X_\infty})
\Longleftrightarrow
\bfE j_{X^\an_\infty!!}(K)\in{}^p\bfE_{\CC\mbox{\scriptsize -}c}^{\geq0}(\I\CC_{\tl{X}})$.
\end{itemize}
\end{lemma}

\begin{proof}
Since the proof of (2) is similar, we shall only prove (1).
Let $K\in\BEC_{\CC\mbox{\scriptsize -}c}(\I\CC_{X_\infty})$.
By Proposition \ref{prop2.12},
we have
$$K\in{}^p\bfE_{\CC\mbox{\scriptsize -}c}^{\leq0}(\CC_{X_\infty}^\sub)
\Longleftrightarrow
I_{X_\infty^\an}^{\rmE}(K)\in {}^p\bfE^{\leq0}_{\CC\mbox{\scriptsize -}c}(\I\CC_{X_\infty}).$$
Moreover, by \cite[Lem.\:3.23]{Ito21} we have
$$I_{X_\infty^\an}^{\rmE}(K)\in {}^p\bfE^{\leq0}_{\CC\mbox{\scriptsize -}c}(\I\CC_{X_\infty})
\Longleftrightarrow
\bfE j_{X^\an_\infty!!} I_{X_\infty^\an}^{\rmE}(K)\in {}^p\bfE^{\leq0}_{\CC\mbox{\scriptsize -}c}(\I\CC_{\tl{X}}).$$
Since, we have an isomorphism in $\BEC(\I\CC_{\tl{X}})$
$$\bfE j_{X^\an_\infty!!} I_{X_\infty^\an}^{\rmE}(K)\simeq
I_{\tl{X}^\an}^{\rmE}\bfE j_{X^\an_\infty!!}(K)$$
by \cite[Lem.\:3.19 (2)(ii)]{Ito24a},
we have 
$$\bfE j_{X^\an_\infty!!} I_{X_\infty^\an}^{\rmE}(K)\in {}^p\bfE^{\leq0}_{\CC\mbox{\scriptsize -}c}(\I\CC_{\tl{X}})
\Longleftrightarrow
\bfE j_{X^\an_\infty!!}(K)\in{}^p\bfE_{\CC\mbox{\scriptsize -}c}^{\leq0}(\CC_{\tl{X}}^\sub)$$
by Proposition \ref{prop2.12}.

The proof is completed.
\end{proof}

Let us recall that the triangulated category $\BEC_{\CC\mbox{\scriptsize -}c}(\CC_{X}^\sub)$
is the full triangulated subcategory of $\BEC_{\CC\mbox{\scriptsize -}c}(\CC_{X^\an}^\sub)$,
see \cite[Cor.\:5.7]{Ito24b} for the details.
\begin{lemma}\label{lem2.20}
Let $X$ be a smooth \underline{complete algebraic variety} over $\CC$.
Then we have
\begin{itemize}
\item[\rm(1)]
${}^p\bfE_{\CC\mbox{\scriptsize -}c}^{\leq 0}(\CC_{X}^\sub) = 
{}^p\bfE_{\CC\mbox{\scriptsize -}c}^{\leq 0}(\CC_{X^\an}^\sub)\cap
\BEC_{\CC\mbox{\scriptsize -}c}(\CC_{X}^\sub)$,

\item[\rm(2)]
${}^p\bfE_{\CC\mbox{\scriptsize -}c}^{\geq 0}(\CC_{X}^\sub)= 
{}^p\bfE_{\CC\mbox{\scriptsize -}c}^{\geq 0}(\CC_{X^\an}^\sub)\cap
\BEC_{\CC\mbox{\scriptsize -}c}(\CC_{X}^\sub)$.
\end{itemize}
\end{lemma}

\begin{proof}
Since the proof of (2) is similar, we shall only prove (1).

Let $K$ be an object of ${}^p\bfE_{\CC\mbox{\scriptsize -}c}^{\leq 0}(\CC_{X}^\sub)$.
Then there exists an object $\M\in \DChol^{\geq0}(\D_X)$
such that $$K\simeq\Sol_X^{\rmE,\sub}(\M)[d_X] (=\Sol_{X^\an}^{\rmE,\sub}(\M^\an))$$
by \cite[Thm.\:5.14]{Ito24b} and Theorem \ref{main-thm-4} (2).
Here, $$(\cdot)^\an\colon\Mod(\D_X)\to\Mod(\D_{X^\an})$$
is the analytification functor.
Since the analytification functor is exact,
we have $\M^\an\in\DChol^{\geq0}(\D_{X^\an})$,
and hence we have $$\Sol_{X^\an}^{\rmE,\sub}(\M^\an)[d_X]
\in{}^p\bfE_{\CC\mbox{\scriptsize -}c}^{\leq 0}(\CC_{X^\an}^\sub)$$
by Theorem \ref{main-thm-2} (2).
Therefore we obtain $$K\in{}^p\bfE_{\CC\mbox{\scriptsize -}c}^{\leq 0}(\CC_{X^\an}^\sub)\cap
\BEC_{\CC\mbox{\scriptsize -}c}(\CC_{X}^\sub).$$

Let $K$ be an object of
${}^p\bfE_{\CC\mbox{\scriptsize -}c}^{\leq 0}(\CC_{X^\an}^\sub)
\cap\BEC_{\CC\mbox{\scriptsize -}c}(\CC_{X}^\sub)$.
Since $K\in\BEC_{\CC\mbox{\scriptsize -}c}(\CC_{X}^\sub)$ there exists an object $\M\in\BDChol(\D_X)$ such that 
$$K\simeq\Sol_X^{\rmE,\sub}(\M)[d_X] \big( = \Sol_{X^\an}^{\rmE,\sub}(\M^\an)[d_X]\big)$$
by \cite[Thm.\:5.14]{Ito24b}.
Since $K\in{}^p\bfE_{\CC\mbox{\scriptsize -}c}^{\leq 0}(\I\CC_{X^\an})$
we have $\M^\an\in\DChol^{\geq0}(\D_{X^\an})$
by Theorem \ref{main-thm-2} (2),
and hence we obtain $\M\in\DChol^{\geq0}(\D_{X})$
because the analytification functor $(\cdot)^\an\colon\Mod(\D_X)\to\Mod(\D_{X^\an})$ is exact and faithful.
Therefore we have $$K\simeq\Sol_X^\rmE(\M)[d_X]\in{}^p\bfE_{\CC\mbox{\scriptsize -}c}^{\leq 0}(\CC_{X}^\sub)$$
by Theorem \ref{main-thm-2} (2).
\end{proof}

Let us recall that any morphism $f\colon X\to Y$ of smooth algebraic varieties induces
a morphism $f_\infty^\an\colon X_\infty^\an\to Y_\infty^\an$ of bordered spaces.
See \cite[\S\:2.3]{Ito21} for the details.

\begin{proposition}
Let $f\colon X\to Y$ be a morphism of smooth algebraic varieties.
We assume that there exists a non-negative integer $d\in \ZZ_{\geq0}$
such that $\dim f^{-1}(y)\leq d$ for any $y\in Y$.
\begin{itemize}
\item[\rm(1)]
For any $K\in{}^p\bfE_{\CC\mbox{\scriptsize -}c}^{\leq0}(\CC_{X_\infty}^\sub)$
we have $\bfE f^\an_{\infty!!} K\in{}^p\bfE_{\CC\mbox{\scriptsize -}c}^{\leq d}(\CC_{Y_\infty}^\sub)$.
\item[\rm(2)]
For any $K\in{}^p\bfE_{\CC\mbox{\scriptsize -}c}^{\geq0}(\CC_{X_\infty}^\sub)$
we have $\bfE f^\an_{\infty\ast} K\in{}^p\bfE_{\CC\mbox{\scriptsize -}c}^{\geq -d}(\CC_{Y_\infty}^\sub)$.
\item[\rm(3)]
For any $L\in{}^p\bfE_{\CC\mbox{\scriptsize -}c}^{\leq0}(\CC_{Y_\infty}^\sub)$
we have $\bfE (f_\infty^\an)^{-1} L\in{}^p\bfE_{\CC\mbox{\scriptsize -}c}^{\leq d}(\CC_{X_\infty}^\sub)$.
\item[\rm(4)]
For any $L\in{}^p\bfE_{\CC\mbox{\scriptsize -}c}^{\geq0}(\CC_{Y_\infty}^\sub)$
we have $\bfE (f_\infty^\an)^! L\in{}^p\bfE_{\CC\mbox{\scriptsize -}c}^{\geq -d}(\CC_{X_\infty}^\sub)$.
\end{itemize}
\end{proposition}
\newpage

\begin{proof}
Since the proofs of these assertions in the proposition are similar,
we shall only prove the assertion (1).

Let $K$ be an object of ${}^p\bfE_{\CC\mbox{\scriptsize -}c}^{\leq0}(\CC_{X_\infty}^\sub)$.
Then we have $\bfE j_{X^\an_\infty!!}(K)\in{}^p\bfE_{\CC\mbox{\scriptsize -}c}^{\leq0}(\CC_{\tl{X}}^\sub)$
by Lemma \ref{lem2.19} (1) and hence
$$\bfE j_{X^\an_\infty!!}K\in {}^p\bfE_{\CC\mbox{\scriptsize -}c}^{\leq 0}(\CC_{\tl{X}^\an}^\sub)\cap
\BEC_{\CC\mbox{\scriptsize -}c}(\CC_{\tl{X}}^\sub)$$
by Lemma \ref{lem2.20} (1).
Let us denote by $\tl{f}\colon \tl{X}\to \tl{Y}$ a morphism of complete algebraic varieties such that $\tl{f}|_X = f$
and denote by $\tl{f}^\an\colon \tl{X}^\an\to \tl{Y}^\an$ the morphism of complex manifolds which is induced by $\tl{f}$.
Then we have
$$\bfE \tl{f}^\an_{!!}\(\bfE j_{X^\an_\infty!!}(K)\)\in
{}^p\bfE_{\CC\mbox{\scriptsize -}c}^{\leq d}(\CC_{\tl{Y}^\an}^\sub)$$
by Proposition \ref{prop2.9} (1).
Since the functor $\bfE \tl{f}^\an_{!!}\colon\BEC(\CC_{\tl{X}^\an}^\sub)\to \BEC(\CC_{\tl{Y}^\an}^\sub)$
preserves the algebraic $\CC$-constructibility,
we have $$\bfE \tl{f}^\an_{!!}\(\bfE j_{X^\an_\infty!!}(K)\)\in
\BEC_{\CC\mbox{\scriptsize -}c}(\CC_{\tl{Y}}^\sub).$$
Hence by the fact that there exists an isomorphism of functors:
$$\bfE j_{Y^\an_\infty!!}\circ \bfE f^\an_{\infty !!}\simeq \bfE \tl{f}^\an_{!!}\circ \bfE j_{X^\an_\infty!!},$$
we have 
$$\bfE j_{Y^\an_\infty!!}\(\bfE f^\an_{\infty !!}(K)\)\simeq \bfE \tl{f}^\an_{!!}\(\bfE j_{X^\an_\infty!!}(K)\)\in
{}^p\bfE_{\CC\mbox{\scriptsize -}c}^{\leq d}(\CC_{\tl{Y}^\an}^\sub)\cap
\BEC_{\CC\mbox{\scriptsize -}c}(\CC_{\tl{Y}}^\sub)$$
and hence
$$\bfE j_{Y^\an_\infty!!}\(\bfE f^\an_{\infty !!}(K)\)\in
{}^p\bfE_{\CC\mbox{\scriptsize -}c}^{\leq d}(\CC_{\tl{Y}}^\sub)$$
by Lemma \ref{lem2.20} (1).
Therefore, we have $$\bfE f^\an_{\infty!!} K\in{}^p\bfE_{\CC\mbox{\scriptsize -}c}^{\leq d}(\CC_{Y_\infty}^\sub)$$
by Lemma \ref{lem2.19} (1).
\end{proof}

\begin{corollary}\label{cor2.22}
Let $Z$ be a locally closed smooth subvariety of $X$.
We denote by $i_{Z^\an_\infty}\colon Z^\an_\infty\to X^\an_\infty$
the morphism of bordered spaces induced by the natural embedding $Z\hookrightarrow X$.
\begin{itemize}
\item[\rm(1)]
For any $K\in{}^p\bfE_{\CC\mbox{\scriptsize -}c}^{\leq0}(\CC_{Z_\infty}^\sub)$
we have $\bfE i_{Z^\an_\infty!!} K\in{}^p\bfE_{\CC\mbox{\scriptsize -}c}^{\leq 0}(\CC_{X_\infty}^\sub)$.
\item[\rm(2)]
For any $K\in{}^p\bfE_{\CC\mbox{\scriptsize -}c}^{\geq0}(\CC_{Z_\infty}^\sub)$
we have $\bfE i_{Z^\an_\infty\ast} K\in{}^p\bfE_{\CC\mbox{\scriptsize -}c}^{\geq 0}(\CC_{X_\infty}^\sub)$.
\item[\rm(3)]
For any $L\in{}^p\bfE_{\CC\mbox{\scriptsize -}c}^{\leq0}(\CC_{X_\infty}^\sub)$
we have $\bfE i_{Z^\an_\infty}^{-1} L\in{}^p\bfE_{\CC\mbox{\scriptsize -}c}^{\leq 0}(\CC_{Z_\infty}^\sub)$.
\item[\rm(4)]
For any $L\in{}^p\bfE_{\CC\mbox{\scriptsize -}c}^{\geq0}(\CC_{X_\infty}^\sub)$
we have $\bfE i_{Z^\an_\infty}^! L\in{}^p\bfE_{\CC\mbox{\scriptsize -}c}^{\geq 0}(\CC_{Z_\infty}^\sub)$.
\end{itemize}

In particular, if $Z$ is open $($resp.\ closed$)$,
then the functor 
$$\bfE i_{Z^\an_\infty}^{-1}\simeq\bfE i_{Z^\an_\infty}^{!}\hspace{10pt}
(\mbox{resp.}\ \bfE i_{Z^\an_\infty!!}\simeq \bfE i_{Z^\an_\infty\ast})$$
is t-exact with respect to the perverse t-structures.
\end{corollary}

\begin{remark}\label{rem2.23}
Let $Z$ be a locally closed smooth subvariety of $X$
and assume that the natural embedding $i_Z\colon Z\hookrightarrow X$ is affine.
Then, functors $\bfE i_{Z^\an_\infty\ast}$ and $\bfE i_{Z^\an_\infty!!}$ induce exact functors
form $\Perv(\CC_{Z_\infty}^\sub)$ to $\Perv(\CC_{X_\infty}^\sub)$:
\[\bfE i_{Z^\an_\infty\ast}, \bfE i_{Z^\an_\infty!!} \colon \Perv(\CC_{Z_\infty}^\sub)\to\Perv(\CC_{X_\infty}^\sub).\]
These follow from  Theorem \ref{main-thm-3}, \cite[Prop.\:3.16 (2)(ii), (3)(i), (4)(ii)]{Ito24a},
 and \cite[Rem.\:3.28]{Ito21}.
\end{remark}

\begin{notation}
For a functor $\mathscr{F}\colon\BEC_{\CC\mbox{\scriptsize -}c}(\I\CC_{X_\infty})\to\BEC_{\CC\mbox{\scriptsize -}c}(\I\CC_{Y_\infty})$,
we set $${}^p\!\mathscr{F} := {}^p\SH^0\circ \mathscr{F} \colon\Perv(\I\CC_{X_\infty})\to\Perv(\I\CC_{Y_\infty}),$$
where ${}^p\SH^0$ is the $0$-th cohomology functor with respect to the perverse t-structure.
For example, for a morphism $f\colon X\to Y$ of smooth algebraic varieties,
we set
\begin{align*}
{}^p\bfE (f^\an_{\infty})^{-1} &:= {}^p\SH^0\circ\bfE (f^\an_{\infty})^{-1}\colon
\Perv(\CC_{Y_\infty}^\sub)\to\Perv(\CC_{X_\infty}^\sub),\\
{}^p\bfE (f^\an_{\infty})^{!} &:= {}^p\SH^0\circ\bfE (f^\an_{\infty})^{!}\colon
\Perv(\CC_{Y_\infty}^\sub)\to\Perv(\CC_{X_\infty}^\sub),\\
{}^p\bfE f^\an_{\infty\ast}&:= {}^p\SH^0\circ\bfE f^\an_{\infty\ast}\colon
\Perv(\CC_{X_\infty}^\sub)\to\Perv(\CC_{Y_\infty}^\sub),\\
{}^p\bfE f^\an_{\infty!!}&:= {}^p\SH^0\circ\bfE f^\an_{\infty!!}\colon
\Perv(\CC_{X_\infty}^\sub)\to\Perv(\CC_{Y_\infty}^\sub).
\end{align*} 
\end{notation}

In this paper, for an object $K\in\BEC(\CC_{X_\infty^\an}^\sub)$,
let us define the support of $K$
by the complement of the union of open subsets $U^\an$ of $X^\an$ such that
$K|_{U_\infty^\an} := \bfE i_{U^\an_\infty}^{-1}K\simeq0$
and denote it by $\supp(K)$.
Namely, we set
$$\supp(K) := \left(\bigcup_{U^\an\underset{\mbox{\tiny open}}{\subset} X^\an,\ K|_{U^\an_\infty} = 0}U^\an\right)^c \hspace{7pt}
\subset X^\an.$$
Note that we have 
\[\bigcup_{U^\an\underset{\mbox{\tiny open}}{\subset} X^\an,\ K|_{U^\an_\infty} = 0}U^\an
 \hspace{7pt} = \hspace{7pt}
 \bigcup_{V^\an\underset{\mbox{\tiny open}}{\subset} X^\an,\ K|_{V^\an} = 0}V^\an.\]
Moreover, for a locally closed smooth subvariety $Z$ of $X$, we set 
$$\Perv_{Z}(\CC_{X_\infty}^\sub) :=  \{K\in\Perv(\CC_{X_\infty}^\sub)\ |\ \supp(K)\subset Z^\an\}.$$

\begin{proposition}\label{prop2.25}
Let $Z$ be a \underline{closed smooth subvariety} of $X$.
Then we have an equivalence of abelian categories$:$
\[\xymatrix@C=75pt{
\Perv_{Z}(\CC_{X_\infty}^\sub)\ar@<0.7ex>@{->}[r]^-{{}^p\bfE i_{Z^\an_\infty}^{-1}}\ar@{}[r]|-{\sim}
&
\Perv(\CC_{Z_\infty}^\sub)
\ar@<0.7ex>@{->}[l]^-{\bfE i_{Z^\an_\infty!!}}.
}\]
Furthermore for any $K\in\Perv_Z(\CC_{X_\infty}^\sub)$
there exists an isomorphism in $\Perv(\CC_{Z_\infty}^\sub)$
$${}^p\bfE i_{Z^\an_\infty}^{-1}K\simeq{}^p\bfE i_{Z^\an_\infty}^{!}K.$$
\end{proposition}

\begin{proof}
This proposition follows from Theorem \ref{main-thm-3}, \cite[Prop.\:3.16 (2)]{Ito24a},
 and \cite[Prop.\:3.30]{Ito21}
\end{proof}

\subsection{Minimal Extensions}\label{sec-minimal}
In this subsection we shall consider simple enhanced perverse subanalytic sheaves on a smooth algebraic variety,
and a counterpart of minimal extensions of algebraic holonomic $\D$-modules.
See e.g., \cite[\S 3.4]{HTT} for the details of 
the definition and properties of minimal extensions of algebraic holonomic $\D$-modules.

Let $X$ be a smooth algebraic variety over $\CC$.
We shall define a notion of simplicity for algebraic enhanced perverse subanalytic sheaves as below.

\begin{definition}\label{main-def-3}
A non-zero algebraic enhanced perverse subanalytic sheaf $K\in\Perv(\CC_{X_\infty}^\sub)$ is called simple
if it contains no subobjects in $\Perv(\CC_{X_\infty}^\sub)$ other than $K$ or $0$.
\end{definition}

Let us recall that
a non-zero algebraic enhanced perverse ind-sheaf $K\in\Perv(\I\CC_{X_\infty})$ is called simple
if it contains no subobjects in $\Perv(\I\CC_{X_\infty})$ other than $K$ or $0$.
See \cite[\S 3.4]{Ito21} for the details.

\begin{proposition}\label{prop2.24}
We have the following:
\begin{itemize}
\item[{\rm (1)}]
For any $K\in\Perv(\CC_{X_\infty}^\sub)$,
$K$ is simple if and only if $I_{X_\infty^\an}^\rmE(K)\in\Perv(\I\CC_{X_\infty})$ is simple.

\item[{\rm (2)}]
For any $K\in\Perv(\I\CC_{X_\infty})$,
$K$ is simple if and only if $J_{X_\infty^\an}^\rmE(K)\in\Perv(\CC_{X_\infty}^\sub)$ is simple.
\end{itemize}
\end{proposition}
\begin{proof}
The assertion (2) follows from the assertion (1) and Theorem \ref{main-thm-3}.
We shall prove the assertion (1).

Let $K\in\Perv(\CC_{X_\infty}^\sub)$ be a simple algebraic enhanced subanalytic sheaf
and $L\in\Perv(\I\CC_{X_\infty})$ a subobject of $I_{X_\infty^\an}^\rmE(K)$
which is not isomorphic to $I_{X_\infty^\an}^\rmE(K)$.
Then $J_{X_\infty^\an}^\rmE(L)$ is a subobject of $J_{X_\infty^\an}^\rmE(I_{X_\infty^\an}^\rmE(K))\simeq K$
in $\Perv(\CC_{X_\infty}^\sub)$ 
which is not isomorphic to $K$
by Theorem \ref{main-thm-3}.
Since $K$ is simple, we have $J_{X_\infty^\an}^\rmE(L)\simeq 0$
and hence $L\simeq 0$ by Theorem \ref{main-thm-3}.
Therefore, $I_{X_\infty^\an}^\rmE(K)$ is simple.

Let $K\in\Perv(\CC_{X_\infty}^\sub)$ and assume that  $I_{X_\infty^\an}^\rmE(K)$ is simple.
We take a subobject $L$ of $K$ which is not isomorphic to $K$.
Since $I_{X_\infty^\an}^\rmE(L)$ is a subobject of $I_{X_\infty^\an}^\rmE(K)$
which is not isomorphic to $I_{X_\infty^\an}^\rmE(K)$ by Theorem \ref{main-thm-3}
and $I_{X_\infty^\an}^\rmE(K)$ is simple,
we have $I_{X_\infty^\an}^\rmE(L)\simeq 0$
and hence $L\simeq 0$ by by Theorem \ref{main-thm-3}.
Therefore, $K$ is simple.

The proof is completed.
\end{proof}

Let us recall that
a non-zero perverse sheaf $\F\in\Perv(\CC_X)$ is called simple
if it contains no subobjects in $\Perv(\CC_X)$ other than $\F$ or $0$.
See \cite{BBD} (also \cite[\S 8.2.2]{HTT}) for the details.

\begin{proposition}
For any simple algebraic perverse sheaf $\SF\in\Perv(\CC_X)$,
the algebraic enhanced perverse subanalytic sheaf $e_{X^\an_\infty}^\sub(\rho_{X_\infty^\an\ast}(\SF)) \in\Perv(\CC_{X_\infty}^\sub)$
is also simple.
\end{proposition}

\begin{proof}
This follows from Proposition \ref{prop2.24} (1), \cite[Props.\:3.4 (4)(i), 3.21]{Ito24a}
and \cite[Prop.\:3.35]{Ito21}.
\end{proof}

We denote by $\ZEC(\CC_{X_\infty^\an}^\sub)$ the heart
with respect to the standard t-structure on $\BEC(\CC_{X_\infty^\an}^\sub)$.
See \cite[Prop.\:3.12]{Ito24a} for the details.
In this paper, 
we shall say that $K\in\ZEC(\CC_{X_\infty^\an}^\sub)$ is an enhanced local system of subanalytic type on $X_\infty$
if for any $x\in X$ there exist an open neighborhood $U\subset X$ of $x$ and a non-negative integer $k$
such that $K|_{U_\infty^\an}\simeq(\CC_{U^\an_\infty}^{\rmE,\sub})^{\oplus k}.$
See \cite[The paragraph before Lem.\:3.17]{Ito24a} for the definition of $\CC_{U^\an_\infty}^{\rmE,\sub}$.
Note that for any enhanced local system $K$ of subanalytic type on $X_\infty$,
there exists an integrable connection $\SL$ on $X$
such that $K\simeq \Sol_{X_\infty}^{\rmE,\sub}(\SL)\in\BEC_{\CC\mbox{\scriptsize -}c}(\CC_{X_\infty}^\sub)$.
In particular $K[d_X]$ is an algebraic enhanced perverse subanalytic sheaf on $X_\infty$. 
Note also that  for any $K\in\ZEC(\CC_{X_\infty^\an}^\sub)$,
$K$ is an enhanced local system of subanalytic type
if and only if $I_{X_\infty^\an}^\rmE(K)$ is an enhanced local system which is defined in
\cite[The paragraph before Prop.\:3.36]{Ito21}.

\begin{proposition}\label{main-thm-5}
\begin{itemize}
\item[\rm(1)]
Let $Z$ be a locally closed smooth connected subvariety of $X$
and assume that the natural embedding $i_Z\colon Z\hookrightarrow X$ is affine.
Then for any simple algebraic enhanced perverse subanalytic sheaf $K$ on $X_\infty$,
the image of the canonical morphism $\bfE i_{Z^\an_\infty!!}K\to\bfE i_{Z^\an_\infty\ast}K$ is also simple,
and it is characterized as the unique simple submodule $($resp.\ unique simple quotient module$)$
of $\bfE i_{Z^\an_\infty\ast}K$ $($resp.\ $\bfE i_{Z^\an_\infty!!}K)$.

\item[\rm(2)]
For any simple algebraic enhanced perverse subanalytic sheaf $K$ on $X_\infty$,
there exist a locally closed smooth connected subvariety $Z$ of $X$ whose natural embedding is affine
and a simple enhanced local system $L$ of subanalytic type on $Z_\infty$ such that
$$K\simeq\Image\big(\bfE i_{Z^\an_\infty!!}L[d_Z]\to\bfE i_{Z^\an_\infty\ast}L[d_Z]\big).$$

\item[\rm(3)]
Let $(Z, L)$ be as in $(1)$ and $(Z', L')$ be another such pair. 
Then we have
$$\Image\big(\bfE i_{Z^\an_\infty!!}L[d_Z]\to\bfE i_{Z^\an_\infty\ast}L[d_Z]\big)\simeq
\Image\big(\bfE i_{Z^\an_\infty!!}L'[d_{Z'}]\to\bfE i_{Z^\an_\infty\ast}L'[d_{Z'}]\big)$$
if and only if
$\var{Z} = \var{Z'}$ and there exists an open dense subset $U$ of $Z\cap Z'$ such that
$L|_{U^\an_\infty}\simeq L'|_{U^\an_\infty}$.
Here, the symbol $\overline{(\cdot)}$ is the closure in $X$.
\end{itemize}
\end{proposition}

\begin{proof}
This proposition follows from  Proposition \ref{prop2.24}, \cite[Prop.\:3.16 (2)(ii), (3)(i), (4)(ii)]{Ito24a},
\cite[Prop.\:3.4]{Ito24b} and \cite[Prop.\:3.36]{Ito21}
\end{proof}

From now on, we shall consider the image of the canonical morphism
\[{}^p\bfE i_{Z^\an_\infty!!}K\to{}^p\bfE i_{Z^\an_\infty\ast}K\]
for a locally closed smooth subvariety $Z$ of $X$
(not necessarily the natural embedding $i_Z\colon Z\hookrightarrow X$ is affine)
and $K\in\Perv(\CC_{Z_\infty}^\sub)$.
Remark that the canonical morphism $${}^p\bfE i_{Z^\an_\infty!!}K\to{}^p\bfE i_{Z^\an_\infty\ast}K$$
is induced by the canonical morphism of functors
$\Perv(\CC_{Z_\infty}^\sub)\to\Perv(\CC_{X_\infty}^\sub)$
\[{}^p\bfE i_{Z^\an_\infty!!}\to{}^p\bfE i_{Z^\an_\infty\ast}\]
and it is an isomorphism if $Z$ is closed. 
In this paper,
we shall define minimal extensions of algebraic enhanced perverse subanalytic sheaves as follows.
\begin{definition}\label{main-def-4}
For any $K\in\Perv(\CC_{Z_\infty}^\sub)$,
we call the image of the canonical morphism $${}^p\bfE i_{Z^\an_\infty!!}K\to{}^p\bfE i_{Z^\an_\infty\ast}K$$
the minimal extension of $K$ along $Z$,
and denote it by ${}^p\bfE i_{Z^\an_\infty!!\ast}K$.
\end{definition}

Since the category $\Perv(\CC_{X_\infty}^\sub)$ is abelian,
the minimal extension ${}^p\bfE i_{Z^\an_\infty!!\ast}K$ of $K\in\Perv(\CC_{Z_\infty}^\sub)$ along $Z$ is also
an algebraic enhanced perverse subanalytic sheaf on $X_\infty$.
Moreover we have a functor
\[{}^p\bfE i_{Z^\an_\infty!!\ast}\colon\Perv(\CC_{Z_\infty}^\sub)\to\Perv(\CC_{X_\infty}^\sub).\]

\begin{remark}\label{rem2.31}
\begin{itemize}
\item[(1)]
If $Z$ is open then we have
$
\bfE i_{Z^\an_\infty}^{-1}\big({}^p\bfE i_{Z^\an_\infty!!\ast}K\big)\simeq K$
by the definition of minimal extensions along $Z$.
\item[(2)]
If $Z$ is closed then we have ${}^p\bfE i_{Z^\an_\infty!!\ast}K
\simeq\bfE i_{Z^\an_\infty\ast}K\simeq\bfE i_{Z^\an_\infty!!}K\in\Perv_Z(\CC_{X_\infty}^\sub)$
by Corollary \ref{cor2.22} and Proposition \ref{prop2.25}.
Namely, in the case when $Z$ is closed,
minimal extensions along $Z$ can be characterized by Proposition \ref{prop2.25}.

\item[(3)]
If the natural embedding $i_Z\colon Z\hookrightarrow X$ is affine,
then we have $${}^p\bfE i_{Z^\an_\infty!!\ast}(\cdot) \simeq
 \Image\big(\bfE i_{Z^\an_\infty!!}(\cdot)\to\bfE i_{Z^\an_\infty\ast}(\cdot)\big)$$
by the fact that the functors $\bfE i_{Z^\an_\infty\ast}$ and $\bfE i_{Z^\an_\infty!!}$ are t-exact
with respect to the perverse t-structures, see Remark \ref{rem2.23} for the details.
\end{itemize}
\end{remark}

Let us recall that 
for any $K\in\Perv(\I\CC_{Z_\infty})$,
we call the image of the canonical morphism $${}^p\bfE i_{Z^\an_\infty!!}K\to{}^p\bfE i_{Z^\an_\infty\ast}K$$
the minimal extension of $K$ along $Z$,
and denote it by ${}^p\bfE i_{Z^\an_\infty!!\ast}K$.
See \cite[Def.\:3.37]{Ito21} for the details.
The following proposition means that
the minimal extension functor ${}^p\bfE i_{Z^\an_\infty!!\ast}$ commutes with functors $I^\rmE$ and $J^\rmE$.

\begin{proposition}\label{prop2.32}
The following diagrams are commutative:
\[\xymatrix@C=45pt@M=5pt{
\Perv(\I\CC_{Z_\infty})\ar@{->}[r]^-{{}^p\bfE i_{Z^\an_\infty!!\ast}}
 & \Perv(\I\CC_{X_\infty})\\
\Perv(\CC_{Z_\infty}^\sub)\ar@{->}[r]_-{{}^p\bfE i_{Z^\an_\infty!!\ast}}
\ar@{->}[u]^-{I^\rmE_{Z^\an_\infty}}
&\Perv(\CC_{X_\infty}^\sub)\ar@{->}[u]_-{I^\rmE_{X^\an_\infty}},}\hspace{33pt}
\xymatrix@C=45pt@M=5pt{
\Perv(\I\CC_{Z_\infty})\ar@{->}[r]^-{{}^p\bfE i_{Z^\an_\infty!!\ast}}
\ar@{->}[d]_-{J^\rmE_{Z^\an_\infty}}
 & \Perv(\I\CC_{X_\infty})\ar@{->}[d]^-{J^\rmE_{X^\an_\infty}}\\
\Perv(\CC_{Z_\infty}^\sub)\ar@{->}[r]_-{{}^p\bfE i_{Z^\an_\infty!!\ast}}
&\Perv(\CC_{X_\infty}^\sub).
}\]
Namely, for any $L\in\Perv(\CC_{Z_\infty}^\sub)$ there exists an isomorphism in $\Perv(\I\CC_{X_\infty})$:
$$I^\rmE_{X^\an_\infty}({}^p\bfE i_{Z^\an_\infty!!\ast}L)
\simeq {}^p\bfE i_{Z^\an_\infty!!\ast}(I^\rmE_{Z^\an_\infty}(L))$$
and for any $K\in\Perv(\I\CC_{Z_\infty})$ there exists an isomorphism in $\Perv(\CC_{X_\infty}^\sub)$:
$$J^\rmE_{X^\an_\infty}({}^p\bfE i_{Z^\an_\infty!!\ast}K)
\simeq {}^p\bfE i_{Z^\an_\infty!!\ast}(J^\rmE_{Z^\an_\infty}(K)).$$
\end{proposition}

\begin{proof}
Let $L\in\Perv(\CC_{Z_\infty}^\sub)$.
Then we have isomorphisms in $\Perv(\I\CC_{X_\infty})$
\begin{align*}
I^\rmE_{X^\an_\infty}({}^p\bfE i_{Z^\an_\infty!!}L)
&\simeq 
I^\rmE_{X^\an_\infty}({}^p\SH^0(\bfE i_{Z^\an_\infty!!}L))\\
&\simeq 
{}^p\SH^0(I^\rmE_{X^\an_\infty}(\bfE i_{Z^\an_\infty!!}L))\\
&\simeq 
{}^p\SH^0(\bfE i_{Z^\an_\infty!!}I^\rmE_{Z^\an_\infty}(L))\\
&\simeq 
{}^p\bfE i_{Z^\an_\infty!!}I^\rmE_{Z^\an_\infty}(L)
\end{align*}
where in the second isomorphism we used the fact that 
the functor $I^\rmE_{X^\an_\infty}$ is t-exact 
with respect to the perverse t-structures by Proposition \ref{prop2.12}
and in the third isomorphism we used \cite[Prop.\:3.16 (2)(ii)]{Ito24a}.
In the similar way, we have an isomorphism in $\Perv(\I\CC_{X_\infty})$:
$$I^\rmE_{X^\an_\infty}({}^p\bfE i_{Z^\an_\infty\ast}L)
\simeq
{}^p\bfE i_{Z^\an_\infty\ast}I^\rmE_{Z^\an_\infty}(L).$$
Therefore we obtain isomorphisms in $\Perv(\I\CC_{X_\infty}^\sub)$:
\begin{align*}
I^\rmE_{X^\an_\infty}({}^p\bfE i_{Z^\an_\infty!!\ast}L)
&\simeq
I^\rmE_{X^\an_\infty}\left(\Image\left({}^p\bfE i_{Z^\an_\infty!!}L\to{}^p\bfE i_{Z^\an_\infty\ast}L\right)\right)\\
&\simeq
\Image\left(I^\rmE_{X^\an_\infty}({}^p\bfE i_{Z^\an_\infty!!}L)\to
I^\rmE_{X^\an_\infty}({}^p\bfE i_{Z^\an_\infty\ast}L)\right)\\
&\simeq
\Image\left({}^p\bfE i_{Z^\an_\infty!!}(I^\rmE_{Z^\an_\infty}(L))\to
{}^p\bfE i_{Z^\an_\infty\ast}(I^\rmE_{Z^\an_\infty}(L))\right)\\
&\simeq
{}^p\bfE i_{Z^\an_\infty!!\ast}I^\rmE_{Z^\an_\infty}(L)
\end{align*}
where
in the second isomorphism we used fact that
the functor $I^\rmE_{X^\an_\infty}$ is t-exact 
with respect to the perverse t-structures by Proposition \ref{prop2.12},
see also \cite[Rem.\:10.1.15]{KS90} for the details.

The last assertion can be proved in the similar way.
We shall skip the proof.
\end{proof}

The following proposition means that 
the minimal extension functor ${}^p\bfE i_{Z^\an_\infty!!\ast}$ commutes
with the duality functor.

\begin{proposition}
The following diagram is commutative:
\[\xymatrix@C=60pt@M=5pt{
\Perv(\CC_{Z_\infty}^\sub)^\op\ar@{->}[r]^-{{}^p\bfE i_{Z^\an_\infty!!\ast}}
\ar@{->}[d]_-{\rmD^\rmE_{Z^\an_\infty}}^-\wr
 & \Perv(\CC_{X_\infty}^\sub)^\op\ar@{->}[d]^-{\rmD^\rmE_{Z^\an_\infty}}_-\wr\\
\Perv(\CC_{Z_\infty}^\sub)\ar@{->}[r]_-{{}^p\bfE i_{Z^\an_\infty!!\ast}}
&\Perv(\CC_{X_\infty}^\sub).
}\]
Namely, for any $K\in\Perv(\CC_{Z_\infty}^\sub)$,
there exists an isomorphism in $\Perv(\CC_{X_\infty}^\sub)$: 
\[\rmD_{X^\an_\infty}^\rmE({}^p\bfE i_{Z^\an_\infty!!\ast}K)
\simeq
{}^p\bfE i_{Z^\an_\infty!!\ast}(\rmD_{Z^\an_\infty}^\rmE K).\]
\end{proposition}

\begin{proof}
This proposition follows from Theorem \ref{main-thm-3}, Proposition \ref{prop2.32},
\cite[Prop.\:3.25 (1), (2)]{Ito24a}, \cite[Thm.\:4.19]{Ito24b} and \cite[Prop.\:3.39]{Ito21}.
\end{proof}

Let us recall that the minimal extension ${}^p\bfR i_{Z!\ast}\F$ of a perverse sheaf $F$ along $Z$ is defined by
the image of the canonical morphism \[{}^p\bfR i_{Z^\an!}\F\to{}^p\bfR i_{Z^\an\ast}\F,\]
where ${}^p\bfR i_{Z^\an!} := {}^p\SH^0\circ\bfR i_{Z^\an!}$,
${}^p\bfR i_{Z^\an\ast} := {}^p\SH^0\circ\bfR i_{Z^\an\ast}$
and ${}^p\SH^0$ is the $0$-th cohomology functor with respect to the perverse t-structures.
See \cite{BBD} (also \cite[\S 8.2.2]{HTT}) for the details.
The following proposition means that the natural embedding functor
commutes with the minimal extension functor.

\begin{proposition}
The following diagram is commutative:
\[\xymatrix@C=60pt@M=5pt{
\Perv(\CC_{Z_\infty}^\sub)\ar@{->}[r]^-{{}^p\bfE i_{Z^\an_\infty!!\ast}}
\ar@{}[rd]|{\rotatebox[origin=c]{180}{$\circlearrowright$}}
 & \Perv(\CC_{X_\infty}^\sub)\\
\Perv(\CC_{Z})\ar@{^{(}->}[u]^-{e_{Z^\an_\infty}^\sub\circ\rho_{Z^\an_\infty\ast}}
\ar@{->}[r]_-{{}^p\bfR i_{Z^\an!\ast}}
&\Perv(\CC_{X})\ar@{^{(}->}[u]_-{e_{X^\an_\infty}^\sub\circ\rho_{X^\an_\infty\ast}}.
}\]
Namely, for any $\SF\in\Perv(\CC_{Z})$,
there exists an isomorphism in $\Perv(\CC_{X_\infty}^\sub):$
\[e_{X^\an_\infty}^\sub(\rho_{X^\an_\infty\ast}({}^p\bfR i_{Z^\an!\ast}\SF))
\simeq
{}^p\bfE i_{Z^\an!!\ast}(e_{Z^\an_\infty}^\sub(\rho_{Z^\an_\infty\ast}(\SF))).\]
\end{proposition}

\begin{proof}
This proposition follows from Theorem \ref{main-thm-3}, Proposition \ref{prop2.32},
 \cite[Prop.\:3.21]{Ito24a} and \cite[Prop.\:3.40]{Ito21}.
\end{proof}

Recall that in the case when $Z$ is a closed smooth subvariety of $X$,
minimal extensions along $Z$ can be characterized
by Proposition \ref{prop2.25}, see also Remark \ref{rem2.31} (2). 
On the other hand, in the case when $Z$ is open whose complement is a smooth subvariety,
the minimal extensions along $Z$ can be characterized as follows.
Let $U$ be such an open subset of $X$ and set $W := X\setminus U$.
Namely $U$ is an open subset of $X$ and $W := X\setminus U$ is a closed smooth subvariety of $X$. 

\begin{proposition}
In the situation as above,
the minimal extension ${}^p\bfE i_{U^\an_\infty!!\ast}K$ of $K\in\Perv(\CC_{U_\infty}^\sub)$ along $U$
is characterized as the unique algebraic enhanced perverse subanalytic sheaf $L$ on $X_\infty$ satisfying the conditions
\begin{itemize}
\setlength{\itemsep}{-2pt}
\item[\rm(1)]
$\bfE i_{U^\an_\infty}^{-1}L\simeq K$,
\item[\rm(2)]
$\bfE i_{W^\an_\infty}^{-1}L\in\bfE^{\leq-1}_{\CC\mbox{\scriptsize -}c}(\CC_{W_\infty}^\sub)$,
\item[\rm(3)]
$\bfE i_{W^\an_\infty}^{!}L\in\bfE^{\geq1}_{\CC\mbox{\scriptsize -}c}(\CC_{W_\infty}^\sub)$.
\end{itemize}
\end{proposition}

\begin{proof}
This proposition follows from Theorem \ref{main-thm-3}, Propositions \ref{prop2.12}, \ref{prop2.32},
 \cite[Prop.\:3.16 (2)(i), (iii)]{Ito24a} and \cite[Prop.\:3.42]{Ito21}.
\end{proof}

Furthermore the minimal extensions along $U$ have following properties.
\begin{proposition}\label{prop2.36}
In the situation as above, for any $K\in \Perv(\CC_{U_\infty}^\sub)$
\begin{itemize}
\item[\rm(1)]
${}^p\bfE i_{U^\an_\infty\ast}K\in\Perv(\CC_{X_\infty}^\sub)$
has no non-trivial subobject in $\Perv(\CC_{X_\infty}^\sub)$ whose support is contained in $W^\an$.
\item[\rm(2)]
${}^p\bfE i_{U^\an_\infty!!}K\in\Perv(\CC_{X_\infty}^\sub)$
has no non-trivial quotient object in $\Perv(\CC_{X_\infty}^\sub)$ whose support is contained in $W^\an$.
\end{itemize}
\end{proposition}

\begin{proof}
This proposition follows from Theorem \ref{main-thm-3}, Propositions \ref{prop2.32} and \cite[Prop.\:3.43]{Ito21}.
\end{proof}

\begin{corollary}
In the situation as above,
the minimal extension ${}^p\bfE i_{U^\an_\infty!!\ast}K$ of $K\in\Perv(\CC_{U_\infty}^\sub)$ along $U$
has neither non-trivial subobject nor non-trivial quotient object in $\Perv(\CC_{X_\infty}^\sub)$
whose support is contained in $W^\an$.
\end{corollary}

\begin{proof}
This corollary follows from Proposition \ref{prop2.36} and fact that
there exist the canonical morphisms in $\Perv(\CC_{X_\infty}^\sub)$:
\[{}^p\bfE i_{U^\an_\infty!!}K\twoheadrightarrow
{}^p\bfE i_{U^\an_\infty!!\ast}K\hookrightarrow
{}^p\bfE i_{U^\an_\infty\ast}K.\]
\end{proof}

Let $Z$ be a locally closed smooth subvariety of $X$
(not necessarily the natural embedding $i_Z\colon Z\hookrightarrow X$ is affine) again.

\begin{proposition}\label{main-thm-6}
We assume that there exist an open subset $U\subset X$ and a closed subvariety $W\subset X$
such that
$Z = U\cap W$ and the complement $X\setminus U$ of $U$ is smooth.
Then we have:

\begin{itemize}
\item[\rm(1)]
%
\begin{itemize}
\item[\rm(i)]
For an exact sequence $0\to K\to L$ in $\Perv(\CC_{Z_\infty}^\sub)$,
the associated sequence $0\to {}^p\bfE i_{Z^\an_\infty!!\ast}K\to {}^p\bfE i_{Z^\an_\infty!!\ast}L$ 
in $\Perv(\CC_{X_\infty}^\sub)$
is also exact.
\item[\rm(ii)]
For an exact sequence $K\to L\to 0$ in $\Perv(\CC_{Z_\infty}^\sub)$,
the associated sequence ${}^p\bfE i_{Z^\an_\infty!!\ast}K\to {}^p\bfE i_{Z^\an_\infty!!\ast}L\to0$
in $\Perv(\CC_{X_\infty}^\sub)$
is also exact.
\end{itemize}

\item[\rm(2)]
For any simple object in $\Perv(\CC_{Z_\infty}^\sub)$,
its minimal extension along $Z$ is also a simple object in $\Perv(\CC_{X_\infty}^\sub)$.
\end{itemize}
\end{proposition}

\begin{proof}
This proposition follows from Theorem \ref{main-thm-3}, Propositions \ref{prop2.12}, \ref{prop2.32}
and \cite[Thm.\:3.47]{Ito21}.
\end{proof}

\end{document}